\documentclass[a4paper]{article}
\usepackage[margin=20mm]{geometry}
\usepackage[utf8]{inputenc}
\usepackage{graphics}
\usepackage{epsfig}
\usepackage{xspace}
\usepackage{url}
\usepackage{amssymb} 

\usepackage{amsmath,amscd,amsthm}
\usepackage{amsfonts}
\usepackage{multirow,booktabs}
\usepackage{amssymb}
\usepackage{hyperref,tikz,tikz-cd}
\usepackage{caption}
\usepackage{accents}
\usepackage{subcaption}
\usepackage{pdfpages}
\usepackage[mathscr]{euscript}
\usepackage{mathrsfs}
\usepackage{listings}
 \usepackage{enumerate}
 \usepackage{cite}
 \usepackage{hyperref}
 \hypersetup{linkcolor=teal, colorlinks=true ,citecolor = cyan}
\usetikzlibrary{decorations.pathreplacing,decorations.markings}

\newtheorem{theorem}{Theorem}[section]
\newtheorem{lemma}[theorem]{Lemma}
\newtheorem{proposition}[theorem]{Proposition}
\newtheorem{corollary}[theorem]{Corollary}
\newtheorem{remark}[theorem]{Remark}
\newtheorem{definition}[theorem]{Definition}

\catcode`@=11 \@addtoreset{equation}{section}
\renewcommand\theequation{\thesection.\@arabic\c@equation}
\catcode`@=12

\makeatother
\usepackage[english]{babel}
\providecommand{\keywords}[1]
{
  \small	
  \textbf{\textit{Keywords---}} #1
}
\providecommand{\MSC}[1]
{
  \small	
  \textbf{\textit{Mathematics Subject Classification---}} #1
}
\usepackage{titlesec,authblk}

 \usepackage[hyperpageref]{backref}

\title{On the study of $(p, Q)$-Laplace Choquard equations with critical Trudinger-Moser nonlinearity in $\mathbb{H}^N$ }
\date{}
\author{Deepak Kumar Mahanta$^{1}$, Tuhina Mukherjee$^{1}$, Abhishek Sarkar$^{1,}$\thanks{Corresponding author} \& Lovelesh Sharma$^{1}$ \\
        \small $^{1}$  Department of Mathematics, Indian Institute of Technology Jodhpur, Rajasthan 342030, India \\
        }
\newcommand{\Addresses}{{
  \bigskip
  \footnotesize

  D.K.~Mahanta, \textit{E-mail address:} \texttt{mahanta.1@iitj.ac.in}

  \medskip

  T.~Mukherjee, \textit{E-mail address:} \texttt{tuhina@iitj.ac.in}

  \medskip

   A.~Sarkar, \textit{E-mail address:} \texttt{abhisheks@iitj.ac.in}

    \medskip

  L.~Sharma, \textit{E-mail address:} \texttt{sharma.94@iitj.ac.in}

}}

\begin{document}
\maketitle \vspace{-1.8\baselineskip}
\begin{abstract}
This paper deals with the existence and multiplicity of nontrivial solutions for  $(p, Q)$-Laplace equations with the Stein-Weiss reaction under critical exponential nonlinearity in the Heisenberg group $\mathbb{H}^N$. In addition, a weight function and two positive parameters have also been included in the nonlinearity. The developed analysis is significantly influenced by these two parameters. Further, the mountain pass theorem, the Ekeland variational principle, the Trudinger-Moser inequality, the doubly weighted  Hardy-Littlewood-Sobolev inequality and a completely new Br\'ezis-Lieb type lemma for Choquard nonlinearity play key roles in our proofs.  
  
\end{abstract}
\\
\noindent\MSC{35D30,35J20,35J60, 35J92,35R03} 
\section{Introduction}
In this paper, we deal with a Choquard equation involving $(p, Q)$-Laplace operator and critical Trudinger-Moser nonlinearity in $\mathbb{H}^N$. More precisely, we study the following equation
\begin{equation}\label{main problem}
  \mathcal{L}_{H,p}(u)+\mathcal{L}_{H,Q}(u)=\mu g(\xi)|u|^{s-2}u+\gamma\Bigg(\displaystyle\int_{\mathbb{H}^N}\frac{F(\eta,u)}{r(\eta)^\beta {d_K(\xi,\eta)}^\lambda}~\mathrm{d}\eta\Bigg)\frac{f(\xi,u)}{r(\xi)^\beta}~~\text{in}~~\mathbb{H}^N,\tag{\textcolor{purple}{$\mathcal{P}_{\mu,\gamma}$}}
\end{equation}
where 
$$\mathcal{L}_{H,t}(u)=-\Delta_{H,t} u+|u|^{t-2}u~~\text{for}~t\in\{p,Q\}$$
with $Q=2N+2$ is the homogeneous dimension of the Heisenberg group $\mathbb{H}^N$. Further, we assume that  $1<s<p<Q,~\beta\geq 0,~0<\lambda<Q$, $2\beta+\lambda<Q$, $\mu$ and $\gamma$ are two positive parameters, and $g:\mathbb{H}^N\to (0,\infty)$ is a weight function in $L^\vartheta(\mathbb{H}^N)$ with $\vartheta=\frac{Q}{Q-s}$. The functions $r(\cdot)$ and $d_K(\cdot,\cdot)$ that appear in \eqref{main problem} are called the Kor\'anyi norm and the Kor\'anyi distance respectively in $\mathbb{H}^N$, which are defined by
$$r(\xi)=r(z,t)=(|z|^4+t^2)^{\frac{1}{4}},~\forall~\xi\in \mathbb{H}^N~~\text{and}~~d_K(\xi,\xi^\prime)=r(\xi^{-1}o ~\xi^\prime),~\forall~(\xi,\xi^\prime)\in \mathbb{H}^N\times \mathbb{H}^N, $$
where $\xi=(z,t)\in \mathbb{H}^N,~z=(x,y)=(x_1,\dots,x_N,y_1,\dots,y_N)\in\mathbb{R}^N\times \mathbb{R}^N,~t\in \mathbb{R}$, and $|z|$ is the Euclidean norm in $\mathbb{R}^{2N}$. Similarly, we can define $\xi^\prime$ as well. Note that $\xi^{-1}=-\xi,$ and  $\xi^{-1}o ~\xi^\prime$ can be computed as in \eqref{eq2.1}.

Moreover, for any $\varphi\in\mathcal{C}^2(\mathbb{H}^N)$, the operator $\Delta_{H,t} \varphi=\text{div}_H(|D_H \varphi|_H^{t-2}D_H \varphi)$ with $t\in\{p, Q\}$ is the standard $t$-Kohn–Spencer Laplace (or horizontal t-Laplace) operator. Here, we denote $D_H \varphi$ as the horizontal gradient of  $\varphi$, that is,
$$ D_H \varphi=(X_1 \varphi,\dots,X_N \varphi,Y_1 \varphi,\dots,Y_N \varphi), $$ where $\{X_j, Y_j\}_{j=1}^{N}$ is the standard basis of the horizontal left-invariant vector fields on $\mathbb{H}^N$ with
$$X_j=\frac{\partial}{\partial x_j}+2y_j \frac{\partial}{\partial t},~~ Y_j=\frac{\partial}{\partial y_j}-2x_j \frac{\partial}{\partial t}~~\text{for}~j=1,\dots,N. $$ The nonlinearity $f:\mathbb{H}^N\times \mathbb{R}\to \mathbb{R} $ has critical exponential growth at infinity, i.e., it behaves like $\exp(\alpha|u|^{\frac{Q}{Q-1}})$ when $|u|\to\infty$ for some $\alpha>0$, which means there exists a positive constant $\alpha_0$ such that the following condition holds:
$$\lim_{|u|\to \infty} |f(\xi,u)| \text{exp}(-\alpha|u|^{\frac{Q}{Q-1}})=\begin{cases}
    0~~~~~\text{if}~\alpha>\alpha_0\\
    +\infty~\text{if}~\alpha<\alpha_0
\end{cases},~ \text{uniformly with respect to} ~\xi\in \mathbb{H}^N.$$
Throughout the paper, we have the following assumptions on the nonlinearity $f:\mathbb{H}^N\times \mathbb{R}\to \mathbb{R} $.
\begin{itemize}
    \item [(f1)]
    $f\in\mathcal{C}^1(\mathbb{H}^N\times \mathbb{R},\mathbb{R})$ such that $f(\cdot,u)=0$ for all $u\leq 0$ and $f(\cdot,u)>0$ for all $u>0$. Further, there exists constant $\alpha_0\in(0,\alpha)$ with the property that for all $\varepsilon>0$, there exists $C_\varepsilon>0$ such that 
    \begin{align*}
     |f(\xi,u)|\leq \varepsilon |u|^{Q-1}+C_\varepsilon \Phi(\alpha_0 |u|^{Q^\prime}),~\forall~(\xi,u)\in \mathbb{H}^N\times \mathbb{R},~\text{where}~\Phi(t)=\exp(t)-\sum_{j=0}^{Q-2}\frac{t^j}{j!}~\text{and}~Q^\prime=\frac{Q}{Q-1};   
    \end{align*}
    \item[(f2)] there exists $\sigma>Q$ such that
    $0<\sigma  F(\xi,u)\leq 2u f(\xi,u)$ for a.e. $\xi\in \mathbb{H}^N$ and any $u>0$, where $F$ is the primitive of $f$ and defined by $F(\xi,u)=\displaystyle\int_{0}^{u}f(\xi,\tau)~\mathrm{d}\tau$,~  $\forall~(\xi,u)\in \mathbb{H}^N\times \mathbb{R}$;
    \item[(f3)] $\partial_u f(\cdot,u)$  exists and $\partial_u f(\cdot,u)=0$ for all $u\leq 0$. Consequently, there exists constant $\alpha_0\in(0,\alpha)$ with the property that for all $\varepsilon>0$, there exists $C_\varepsilon>0$ such that 
    $$|\partial_u f(\xi,u) u|\leq \varepsilon |u|^{Q-1}+C_\varepsilon \Phi(\alpha_0 |u|^{Q^\prime}),~\forall~(\xi,u)\in \mathbb{H}^N\times \mathbb{R};$$
    \item[(f4)] there exists $\upsilon>Q$ and a constant $\widehat{C}>0$ such that $$F(\xi,u)\geq \frac{2\widehat{C}}{\upsilon}u^\upsilon	,~\forall~(\xi,u)\in \mathbb{H}^N\times [0,\infty). $$
\end{itemize}
\begin{remark}
A typical example of a function satisfying $(f1)-(f4)$ can be considered as $f(\xi,u)=(u^+)^{Q-1}\Phi\big((u^+)^{Q^{\prime}}\big),$ $\forall~(\xi,u)\in  \mathbb{H}^N\times \mathbb{R}$ with $\alpha_0>1$, where $Q=2N+2$, $Q^{\prime}=\frac{Q}{Q-1}$, $u^+=\max\{u,0\}$ and $\Phi$ is defined as in $(f1)$.
\end{remark}

Over the last decade, analysis of PDEs in the Heisenberg group has gained much attention from many researchers. Of all the noncommutative nilpotent Lie groups, the Heisenberg group is the most straightforward example. The Heisenberg group have particularly played an important role in quantum physics, ergodic theory, representation theory of nilpotent Lie groups, harmonic analysis, differential geometry, several complex variables and CR geometry. Due to the analytical non-Euclidean nature of this space despite its topological Euclidean nature, certain fundamental concepts of analysis, including dilatations, must be developed again ( see \cite{MR4029564}). Folland and Stein \cite{MR0367477} were the pioneering mathematicians who initiated the research of subelliptic analysis on the Heisenberg group. Later, Rothschild and Stein made significant advancements in generalising these results to suit the criteria of Hörmander for vector fields. Furthermore, we also refer the interested readers to study \cite{MR4289241, MR0344700, MR0657581, MR3797740, MR4475683} for more advanced properties of elliptic operators on the Heisenberg group.

In the context of the Heisenberg group, to familiarize the reader with the borderline case of the Sobolev embedding, which is commonly known as the Trudinger-Moser case, we first recall that for $p<Q$, by the classical Sobolev embedding results and the notations used in \eqref{eq2.999}, we have $HW^{1,p}(\mathbb{H}^N) \hookrightarrow L^q (\mathbb{H}^N)$ for $p\leq q\leq p^\ast $, where $p^\ast=\frac{Qp}{Q-p}$, called critical Sobolev exponent. Due to these facts, to study the variational problems with subcritical and critical growth, we mean that the nonlinearity cannot exceed the polynomial of degree $p^\ast$. Further, for the limiting case $p=Q$, we also have $HW^{1,Q}(\mathbb{H}^N) \hookrightarrow L^s (\mathbb{H}^N)$ for $Q\leq s<\infty $ but we cannot take $s=\infty$ for such an embedding. Now, it is fairly natural to ask `` Is there another kind of maximal growth in this situation''? This question was answered by
N. Lama et al. in \cite{MR2980499}. The answer to this question is that the exponential growth can be admitted in the borderline case.

It is worthwhile to mention that there has been a lot of advancement in the Trudinger-Moser-type inequalities to study the existence, nonexistence and multiplicity of solutions of the nonlinear PDEs in the whole Euclidean framework as well as in the Heisenberg group framework. For a detailed study, one can go through \cite{MR3073316, MR2980499, MR2927116, MR2976058, MR3130507, MR3797740, MR4322564,  MR2669653, MR2863858, MR4451907,MR4183246} and references therein. 

Define the functional of the form 
$$ u\mapsto\int_{\mathbb{R}^N} \big(|\nabla u|^p+|\nabla u|^q\big)~\mathrm{d}x, $$
where $1<p<q$. Such types of functionals fall in the realm of the so-called functionals with nonstandard growth conditions of $(p,q)$-type. To the best of our knowledge, the study of this kind of unbalanced integral functionals was first introduced by P. Marcellini in his celebrated research papers \cite{MR1094446, MR0868523}. Since, then elliptic problems involving $(p,q)$-Laplace operator have been extensively studied by many authors due to the broad applications in biophysics, plasma physics, solid state physics, and chemical reaction design (see  \cite{ MR0901723, MR1003258, MR1313718}). In this regard, we refer to  \cite{MR3886584, MR2524439,  MR4029564} and references therein. Despite this, in the borderline case ( $q=Q~(\text{or}~N)$), there are very few contributions in the literature driven by $(p, Q)~(\text{or}~(p, N))$-Laplace operator in the sense Trudinger-Moser nonlinearity, which is one of the key motivation towards studying this paper.  In that context, we want to mention \cite{MR4183246, MR4451907, MR4258779, MR4149293, MR4289110} and related references. 

Recently, the study of nonlinear PDEs involving Hartee-type nonlinearity, commonly known as Choquard-type nonlinearity, which is driven by the classical Hardy-Littlewood-Sobolev inequality ( HLS inequality in short ) \cite{MR1817225} has attracted a lot of researchers due to its wide range of applications in physics such as the Bose-Einstein condensation and the self-gravitational collapse of a quantum mechanical wave function. Initially, S.I.Pekar \cite{pekar1954untersuchungen} used the Hartree-Fock theory to explain the quantum mechanics of a polaron that was at rest. Moreover, P.Choquard \cite{MR0471785} described the model of an electron trapped in its own hole using such kind of nonlinearity. For more details on the study of Choquard-type equations using variational methods, we refer to \cite{MR3356947, MR3804260} and the related references. Further, in the context of the Heisenberg group, for the study of Choquard-type equations, we refer to \cite{MR4475683, MR4060090}. 

On the other hand, nonlinear problems driven by the doubly weighted Hardy- Littlewood-Sobolev inequality (also called Stein-Weiss type inequality) \cite{MR0098285} with critical exponential nonlinearity are very limited in the current literature. For instance, we refer the interested readers to study \cite{MR4062337, MR4683005, MR4564507, MR4626503}. This is another key motivation towards the study of this paper.  

The main features and novelty of this paper are listed as follows: $(a)$ it is worth noting that, due to the unbounded nature of the domain, the nonhomogeneity of the leading operator in \eqref{main problem}, and the lack of compactness of the solution space $\mathbf{X}$ into some Lebesgue space $L^s(\mathbb{H}^N)$, we have some additional difficulty in proving the compactness property of the Palais-Smale sequence, and $(b)$ the appearance of the Stein-Weiss convolution term  under critical exponential nonlinearity in \eqref{main problem} makes our analysis more challenging because for $u_n\rightharpoonup u$ weakly in $\mathbf{X}$ as $n\to\infty$, then we have to verify 
    \begin{align}\label{eq1.25}
    \lim_{n\to\infty} \int_{\mathbb{H}^N}\Bigg(\displaystyle\int_{\mathbb{H}^N}\frac{F(\eta,u_n)}{r(\eta)^\beta {d_K(\xi,\eta)}^\lambda}\mathrm{d}\eta\Bigg)\frac{f(\xi,u_n)v}{r(\xi)^\beta} ~\mathrm{d}\xi=\int_{\mathbb{H}^N}\Bigg(\displaystyle\int_{\mathbb{H}^N}\frac{F(\eta,u)}{r(\eta)^\beta {d_K(\xi,\eta)}^\lambda}\mathrm{d}\eta\Bigg)\frac{f(\xi,u)v}{r(\xi)^\beta}~ \mathrm{d}\xi,\forall~v\in \mathbf{X},
\end{align}
 to confirm that the weak limit  $u$ in the above sequence is the solution to \eqref{main problem}. However, in the case of exponential growth, \eqref{eq1.25} may not hold even we have $u_n\to u$ in $L^\theta(B_R)$ for any $R>0$ and $\theta\in[1,\infty)$ as $n\to\infty$. To prove this we need  Theorem \ref{thm2.7} with the critical exponent $\alpha_0 $ to be less than $\alpha_Q$.

 Motivated by all the aforementioned works, especially by \cite{MR4451907, MR4258779, MR4062337}, we study for the first time in literature the existence of nontrivial nonnegative solutions for \eqref{main problem} involving the Stein-Weiss reaction under critical exponential growth in the Heisenberg group context. Before we state our main results, we first define the weak solution of \eqref{main problem}.
\begin{definition}
    We say that $u\in\mathbf{X}$ (see \eqref{eq2.4}) is a weak solution for \eqref{main problem} if for all $v\in\mathbf{X}$, we have 
$$\big\langle{u,v}\big\rangle_{H,p}+\big\langle{u,v}\big\rangle_{H,Q}=\mu\displaystyle\int_{\mathbb{H}^N} g(\xi)|u|^{s-2}uv~\mathrm{d}\xi+\gamma \displaystyle\int_{\mathbb{H}^N}\Bigg(\displaystyle\int_{\mathbb{H}^N}\frac{F(\eta,u)}{r(\eta)^\beta {d_K(\xi,\eta)}^\lambda}~\mathrm{d}\eta\Bigg)\frac{f(\xi,u)}{r(\xi)^\beta} v~\mathrm{d}\xi,$$
where for $t\in\{p,Q\}$, we define
\begin{align}\label{eq1.1}
\big\langle{u,v}\big\rangle_{H,t}=\int_{\mathbb{H}^N}\Big(|D_H u|_H^{t-2}(D_H u,D_H v)_H+|u|^{t-2}uv\Big)~\mathrm{d}\xi.    
\end{align}
\end{definition}
The main results of this paper can be stated as follows.
\begin{theorem}\label{thm1.1}
 Assume that $1<s<p<Q<\infty$ and $g$ is in $L^\vartheta(\mathbb{R}^N)$ with $\vartheta=\frac{Q}{Q-s}$. Let $(f1)$ and $(f2)$are hold, then there exists $\widehat{\mu}>0$ such that \eqref{main problem} admits at least one nontrivial nonnegative solution $u_{\mu,\gamma}$ in $\mathbf{X}$ for all $\mu\in (0,\widehat{\mu})$ and for all $\gamma>0$. Consequently, there holds $$\lim_{\mu\to 0^+}\|u_{\mu,\gamma}\|_{\mathbf{X}}=0. $$
\end{theorem}
\begin{theorem}\label{thm1.2}
Let all the assumptions of Theorem \ref{thm1.1} be satisfied. Additionally, if $(f3)$ and $(f4)$ are hold, then there exists $\gamma^\ast>0$ such that for all $\gamma>\gamma^\ast$ there exists $\widetilde{\mu}=\widetilde{\mu}(\gamma)>0$ with the property that \eqref{main problem} admits a nontrivial nonnegative solution $w_{\mu,\gamma}$ for all $\mu\in (0,\widetilde{\mu}]$. Consequently, if $\mu<\textit{min}\{\widehat{\mu},\widetilde{\mu}\}$, then $w_{\mu,\gamma}$ is a second solution of \eqref{main problem} which is independent of $u_{\mu,\gamma}$ constructed in Theorem \ref{thm1.1}.  
\end{theorem}

The paper is structured in the following way. In Section \ref{sec2}, we present some basic definitions and properties of the Heisenberg group $\mathbb{H}^N$, as well as the suitable function spaces, and some technical lemmas are also recalled. In Section \ref{sec3}, we provide the variational structure of \eqref{main problem} and prove Theorem \ref{thm1.1} using a minimization argument based on the Ekeland variational principle. Finally, in Section \ref{sec4}, we establish the key compactness lemma, particularly helpful to apply the mountain pass theorem at a suitable mountain pass level and to prove Theorem \ref{thm1.2}.\\ 
\textbf{Notations.} Throughout this paper, we use the following notations:
\begin{itemize}
    \item [$\bullet$] $o_n(1)$ denotes a real sequence such that $o_n(1)\to 0$ as $n\to\infty$.
    \item [$\bullet$] $\rightharpoonup $ means weak convergence and $\rightarrow$ means strong convergence.
    \item [$\bullet$] $u^+=$ max~$\{u,0\}$ and $u^-=$ max~$\{-u,0\}$.
    \item [$\bullet$] $B_R=\{u\in \mathbf{X}:\|u\|_{\mathbf{X}}<R\}$, $\overline{B}_R=\{u\in \mathbf{X}:\|u\|_{\mathbf{X}}\leq R\}$, and $\partial B_R=\{u\in \mathbf{X}:\|u\|_{\mathbf{X}}=R\}$ for any $R>0$.
\end{itemize}

\section{Preliminaries}\label{sec2}
In this section, we first introduce some basic properties of the Heisenberg group $\mathbb{H}^N$ as well as the horizontal Sobolev spaces and some technical lemmas which will be used in the next sections. For more details on the Heisenberg group, we refer to \cite{MR2153462, MR3842328, MR4029564, MR3090536, MR4355761, MR4475683, MR4466642, MR2976058} and references found therein.

The Heisenberg group $\mathbb{H}^N$ is of topological dimension $2N+1$, i.e., the Lie group whose underlying manifold is $\mathbb{R}^{2N+1}$, endowed with the non-abelian group law
\begin{align}\label{eq2.1}
    \xi o~ \xi^\prime=\bigg(z+z^\prime,t+t^\prime+2\displaystyle\sum_{i=1}^{N}(y_ix_i^\prime-x_iy_i^\prime)\bigg),~\forall~\xi,~  \xi^\prime\in \mathbb{H}^N,
\end{align}
where $\xi=(z,t)=(x_1,\dots,x_N,y_1,\dots,y_N,t)$ and $\xi^\prime=(z^\prime,t^\prime)=(x^\prime_1,\dots,x^\prime_N,y^\prime_1,\dots,y^\prime_N,t^\prime)$. Further, we also have $\xi^{-1}=-\xi$ and there holds $(\xi o~ \xi^\prime)^{-1}=(\xi^\prime)^{-1} o~\xi^{-1}$. Here, we denote $\{X_j, Y_j\}_{j=1}^{N}$ as the standard basis of the horizontal left-invariant vector fields on $\mathbb{H}^N$ with
$$X_j=\frac{\partial}{\partial x_j}+2y_j \frac{\partial}{\partial t},~~ Y_j=\frac{\partial}{\partial y_j}-2x_j \frac{\partial}{\partial t},~T=\frac{\partial}{\partial t}~\text{for}~j=1,\dots,N. $$ This basis satisfies the Heisenberg canonical commutation relations for position and momentum with
$$[X_j,Y_k]=-4 \delta_{jk}T,~[X_j,X_k]=[X_j,T]=[Y_j,Y_k]=[Y_j,T]=0.$$
Moreover, a left-invariant vector field $X$, which is in the span of $\{X_j, Y_j\}_{j=1}^{N}$, is called horizontal. Next, for any real number $R$, we define a dilation $\delta_R:\mathbb{H}^N\to\mathbb{H}^N$ by 
\begin{align}\label{eq2.2}
 \delta_R(\xi)=(Rz,R^2t),~\forall~\xi=(z,t)\in \mathbb{H}^N.   
\end{align}
It is easy to see that the Jacobian determinant of dilation $\delta_R$ coincides with $R^{2N+2}$, where the natural number $Q=2N+2$ is the homogeneous dimension of $\mathbb{H}^N$. Also, the Kor\'anyi norm in $\mathbb{H}^N$ is given by
$$r(\xi)=r(z,t)=(|z|^4+t^2)^{\frac{1}{4}},~\forall~\xi=(z,t)\in \mathbb{H}^N.$$
The corresponding distance is called Kor\'anyi distance and is defined by 
$$ d_K(\xi,\xi^\prime)=r(\xi^{-1}o ~\xi^\prime),~\forall~(\xi,\xi^\prime)\in \mathbb{H}^N\times \mathbb{H}^N. $$
This distance acts like the Euclidean distance in horizontal directions and behaves like the square root of the Euclidean distance in the missing direction. One can easily observe that the  Kor\'anyi norm is homogeneous of degree 1 with respect to the dilation. Indeed, for any $R>0$, we have
$$ r(\delta_R(\xi))=r(Rz,R^2t)=(|Rz|^4+R^4t^2)^{\frac{1}{4}}=R~ r(\xi),~\forall~\xi=(z,t)\in \mathbb{H}^N.$$
 The Kor\'anyi open ball in $\mathbb{H}^N$  of radius $R$ and centered at $\xi_0$ is defined as follows:
 $$B_R(\xi_0)=\{\xi\in \mathbb{H}^N:~d_K(\xi,\xi_0)<R \}.$$ For simplicity, we denote $B_R$ as the open ball with centered at $O$ and radius $R$, where $O=(0,0)$ is the natural origin of $\mathbb{H}^N$.

 Under the left translations of the Heisenberg group, the Lebesgue measure on $\mathbb{R}^{2N+1}$ remains invariant. Therefore, $\mathrm{d}\xi$ is the unique Haar measure on $\mathbb{H}^N$ that corresponds with the $(2N+1)$-Lebesgue measure and $|U|$ is the unique $(2N+1)$-dimensional Lebesgue measure of every measurable set $U\subset \mathbb{H}^N$ since the Haar measures on Lie groups are unique up to constant multipliers. Moreover, the Haar measure on $\mathbb{H}^N$ is $Q$-homogenous with respect to dilation $\delta_R$. Consequently, we have
$$|\delta_R(U)| = R^Q|U| ~\text{and} ~\mathrm{d}(\delta_R \xi) = R^Q\mathrm{d}\xi .$$
In particular, we obtain $|B_R|=R^Q|B_1|$. Now, let $u\in\mathcal{C}^1(\mathbb{H}^N,\mathbb{R})$ be fixed, then the horizontal gradient (or intrinsic gradient ) of $u$ is defined as follows:
$$
 D_H u=\sum_{j=1}^{N} [(X_ju)X_j+ (Y_ju)Y_j].   
$$
Observe that $D_H u$ is an element of the span of $\{X_j, Y_j\}_{j=1}^{N}$, denoted by span $\{X_j, Y_j\}_{j=1}^{N}$. Consequently, if $\tilde{g}\in \mathcal{C}^1(\mathbb{R)}$, then $$D_H \tilde{g}(u) =\tilde{g}^\prime(u) D_H u. $$In span$\{X_j, Y_j\}_{j=1}^{N}\cong \mathbb{R}^{2N}$, we
consider the natural inner product given by
$$
  ( X , Y ) _H = \sum_{j=1}^{N} ( x _j y _j + \widetilde{x} _j \widetilde{y}_ j )  
$$
for $X=\{x_jX_j +\widetilde{x}_jY_j\}_{j=1}^N$ and $Y =\{y_jX_j+\widetilde{y}_jY_j\}_{j=1}^N$.The inner product $(\cdot,\cdot)_H$ produces the Hilbertian norm $$|X|_H=\sqrt{(X,X)_H} $$
for the horizontal vector field X. Further, the Cauchy–Schwarz inequality
$$|( X , Y ) _H|\leq |X|_H|Y|_H $$
also valid for any horizontal vector fields $X$ and $Y$. Similarly, for any horizontal vector field function $X=X(\xi)$, $X=\{x_jX_j +\widetilde{x}_jY_j\}_{j=1}^N$, of class $\mathcal{C}^1(\mathbb{H}^N,\mathbb{R}^{2N})$, we define the
horizontal divergence of $X$ by
$$ \text{div}_H X= \sum_{j=1}^{N}[X_j(x_j)+Y_j(\widetilde{x}_j)].$$
For any $u\in \mathcal{C}^1(\mathbb{H}^N)$, then the Leibnitz formula is also valid, that is, 
$$ \text{div}_H (uX)=u~ \text{div}_H X+(D_H u,X)_H.$$
Similarly, if $u\in \mathcal{C}^2(\mathbb{H}^N)$, then the Kohn–Spencer Laplace ( or the horizontal Laplace, or the
sub-Laplace) operator, in $\mathbb{H}^N$, of u is defined by
$$\Delta_H u=\displaystyle\sum_{j=1}^{N}(X_j^2+Y_j^2)u=\displaystyle\sum_{j=1}^{N}\bigg(\frac{\partial^2}{\partial x_j^2}+\frac{\partial^2}{\partial y_j^2}+4y_j \frac{\partial^2}{\partial x_j\partial t}-4x_j\frac{\partial^2}{\partial y_j\partial t}\bigg)u+4|z|^2\frac{\partial^2 u}{\partial t^2}. $$
Due to Hörmander's famous theorem [\cite{MR0222474}, Theorem 1.1], the operator $\Delta_H$ is hypoelliptic. For each $u\in \mathcal{C}^2(\mathbb{H}^N)$, we have $\Delta_H u=\text{div}_H (D_H u)$. A common generalisation of the Kohn-Spencer Laplace operator is the horizontal $p$-Laplace operator for any $p\in(1,\infty)$ on the Heisenberg group, which is  defined by
$$\Delta_{H,p} \varphi=\text{div}_H(|D_H \varphi|_H^{p-2}D_H \varphi),~\forall~\varphi\in \mathcal{C}^\infty_c(\mathbb{H}^N). $$
Let $1\leq\wp<\infty$, then the Lebesgue space $L^\wp(\mathbb{H}^N)$ is defined by
$$ L^\wp(\mathbb{H}^N)=\bigg\{u:\mathbb{H}^N\to \mathbb{R}~\text{is measurable} :~\int_{\mathbb{H}^N}|u|^\wp~\mathrm{d}\xi<\infty \bigg\}, $$
endowed with the norm 
$$\|u\|_\wp=\bigg(\int_{\mathbb{H}^N}|u|^\wp~\mathrm{d}\xi\bigg)^{\frac{1}{\wp}}. $$
Further, if $\wp=\infty$, then the norm for Lebesgue space $L^\infty(\mathbb{H}^N)$ is given by
$$\|u\|_\infty=\inf\big\{M : |u(\xi)|\leq M~\text{for a.e.} ~\xi\in\mathbb{H}^N\big\}.$$
It is well-known that $(L^\wp(\mathbb{H}^N),\|\cdot\|_\wp)$ is a reflexive and separable Banach space for $1<\wp<\infty$. Moreover, $\mathcal{C}^\infty_c(\mathbb{H}^N)$ is a dense subset of $L^\wp(\mathbb{H}^N)$ for any $1\leq \wp<\infty$ (see \cite{MR4029564} ). Also, for any $\Omega\subset \mathbb{H}^N$, we denote the norm of $L^\wp(\Omega)$ by $\|\cdot\|_{\wp,\Omega}$.

Denote
\begin{equation}\label{eq2.999}
  HW^{1,\wp}(\mathbb{H}^N)=\big\{ u\in L^\wp(\mathbb{H}^N):|D_H u|_H\in L^\wp(\mathbb{H}^N)\big\},   
\end{equation}
endowed with the norm
$$ \|u\|_{HW^{1,\wp}}=(\|u\|^\wp_\wp+\|D_H u\|^\wp_\wp)^{\frac{1}{\wp}},~\forall~\wp\in[1,\infty).$$
Notice that the space $(HW^{1,\wp}(\mathbb{H}^N),\|\cdot\|_{HW^{1,\wp}})$ is a reflexive and separable Banach space. Consequently, $\mathcal{C}^\infty_c(\mathbb{H}^N)$ is a dense subset of $HW^{1,\wp}(\mathbb{H}^N)$ (see \cite{MR4029564} ).

Moreover, we define the convolution in the Heisenberg group $\mathbb{H}^N$ (see \cite{MR0657581,MR2458901}). For this, if $u\in L^1(\mathbb{H}^N)$ and $v\in L^\wp(\mathbb{H}^N)$, where $\wp\in[1,\infty)$, then, for a.e. $\xi\in \mathbb{H}^N$, the function $\eta\mapsto u(\xi~o~\eta^{-1})v(\eta)$ is in $L^1(\mathbb{H}^N)$. Moreover, the convolution of $u$ and $v$ is defined as follows:
$$ (u * v)(\xi) = \int_{\mathbb{H}^N} u(\xi~o~\eta^{-1})v(\eta)~\mathrm{d}\eta~~\text{for a.e.} ~\xi\in \mathbb{H}^N .$$
By the analogue of the Young theorem, we obtain
$$ u * v \in L^\wp(\mathbb{H}^N)~~\text{and}~~\|u * v \|_\wp\leq \|u\|_1\|v\|_\wp.$$
Just like in the Euclidean context, the technique of regularization using the convolution can be extended to the Heisenberg group $\mathbb{H}^N$. In particular, it is possible to generate a sequence of mollifiers $(\rho_n)_n $ on $\mathbb{H}^N$ with the properties that $\rho_n \in \mathcal{C}^\infty_c(\mathbb{H}^N), ~\rho_n\geq 0$ in $\mathbb{H}^N$ and $\displaystyle\int_{\mathbb{H}^N} \rho_n(\xi)~\mathrm{d}\xi=1$ for any $n\in\mathbb{N}$, for more details, see \cite{MR4029564}. Consequently, if $\wp\in[1,\infty)$, then for all $u\in HW^{1,\wp}(\mathbb{H}^N)$, we have
$$\rho_n*u\to u~~\text{in}~~L^\wp(\mathbb{H}^N)~~\text{and}~D_H(\rho_n*u)\to D_H u ~~\text{in}~~L^\wp(\mathbb{H}^N,\mathbb{R}^{2N})~\text{as}~n\to\infty.$$

To study \eqref{main problem}, we consider the following function space as our solution space: 
\begin{equation}\label{eq2.4}
 \mathbf{X}=HW^{1,p}(\mathbb{H}^N)\cap HW^{1,Q} (\mathbb{H}^N) ~\text{for} ~1<p<Q,  
\end{equation}
endowed with the norm
$$ \|u\|_{\mathbf{X}}=\|u\|_{HW^{1,p}}+\|u\|_{HW^{1,Q}}. $$
It is easy to see that the function space $(\mathbf{X},\|\cdot\|_{\mathbf{X}})$ is a reflexive and separable Banach space. Observe that the embedding 
     $\mathbf{X}\hookrightarrow HW^{1,r}(\mathbb{H}^N)$ is continuous, where $r\in\{p,Q\}$.\\
Now we list some classical embedding results for horizontal Sobolev space in the Heisenberg group.     
\begin{lemma}[cf. \cite{MR4029564}]\label{lem2.1}
If $q\in[p,p^\ast]$ with $p^\ast=\frac{Qp}{Q-p}$, then the  embedding $HW^{1,p}(\mathbb{H}^N) \hookrightarrow L^q (\mathbb{H}^N)$ is continuous. In addition, the embedding $HW^{1,p}(\mathbb{H}^N)\hookrightarrow L^q (B_R)$ is compact for any $R>0$ and $q\in[1,p^\ast)$.
\end{lemma}
\begin{lemma}[cf.\cite{MR3219231}]\label{lem2.2}
If $s\in[Q,\infty)$, then the  embedding $HW^{1,Q}(\mathbb{H}^N) \hookrightarrow L^s (\mathbb{H}^N)$ is continuous. Furthermore, the embedding $HW^{1,Q}(\mathbb{H}^N)\hookrightarrow L^s (B_R)$ is compact for any $R>0$ and $s\in[1,\infty)$.
\end{lemma}
\begin{corollary} \label{cor2.3}
In light of Lemma \ref{lem2.1} and Lemma \ref{lem2.2}, the  embedding
 $\mathbf{X}\hookrightarrow L^\vartheta (\mathbb{H}^N)$ is continuous for any $\vartheta \in [p,p^\ast]\cup [Q,\infty)$. Moreover, the embedding $\mathbf{X}\hookrightarrow L^\vartheta(B_R)$ is compact for any $R>0$ and $\vartheta\in[1,\infty)$.
\end{corollary}
\begin{lemma}\label{lem2.4}
Let $\vartheta=\frac{Q}{Q-s}$ and $s\in(1,Q)$. Then the embedding $L^Q(\mathbb{H}^N)\hookrightarrow L^s_g(\mathbb{H}^N)$ is continuous for any $s\in(1,Q)$ and there holds $$\|u\|_{s,g}\leq \|h\|^{1/s}_\vartheta \|u\|_Q,~\forall~u\in L^Q(\mathbb{H}^N)~,~\text{where}~\|u\|^s_{s,g}=\int_{\mathbb{H}^N} g(\xi)|u|^s~\mathrm{d}\xi $$
is a norm for the weighted Lebesgue space $L^s_g(\mathbb{H}^N)$ with weight function $g\in L^\vartheta(\mathbb{H}^N)$.

In addition, due to Lemma \ref{lem2.2}, the embedding $HW^{1,Q}(\mathbb{H}^N)\hookrightarrow L^s_g (\mathbb{H}^N)$ is compact. Furthermore, by using the continuous embedding $ \mathbf{X}\hookrightarrow HW^{1,Q}(\mathbb{H}^N)$ , we also conclude that the embedding $\mathbf{X}\hookrightarrow L^s_g (\mathbb{H}^N)$ is compact.
\end{lemma}
\begin{proof}
    By applying the same idea as in [\cite{MR4258779}, Lemma 2.2], the lemma can be proved. 
\end{proof}
\begin{corollary}[cf. \cite{MR2980499}]\label{cor2.5}
The function 
   $\Phi(t)=\exp(t)-\displaystyle\sum_{j=0}^{Q-2}\frac{t^j}{j!}$ is increasing and convex in $[0,\infty)$. Moreover, for any $\wp\geq 1$ and $t\geq 0$ be real numbers, then 
 $$ \bigg(\exp(t)-\sum_{j=0}^{Q-2}\frac{t^j}{j!}\bigg)^\wp\leq \exp(\wp t)-\sum_{j=0}^{Q-2}\frac{(\wp t)^j}{j!}. $$
\end{corollary}
The following theorem was proved by W. S. Cohn et al. in \cite{MR2927116}, called Trudinger–Moser inequality in the Heisenberg group $\mathbb{H}^N$.
\begin{theorem}(\textbf{Trudinger–Moser inequality in $\mathbb{H}^N$-I })\label{thm2.6}
 For any fixed $\alpha>0$, $0\leq \beta<Q$, and $u\in HW^{1,Q}(\mathbb{H}^N)$, then there holds $$\displaystyle\frac{\Phi(\alpha |u|^{Q^\prime})}{r(\xi)^\beta}\in L^1(\mathbb{H}^N),$$ where $\Phi$ is defined as in $(f1)$. Further, if $\alpha_Q=Q\sigma_Q^{\frac{1}{Q-1}}$, where $\sigma_Q=\displaystyle\int_{r(\xi)=1}|z|^Q~\mathrm{d}\mu$ with $\xi=(z,t)\in \mathbb{H}^N$ and let $\alpha^\ast$ be such that $\alpha^{\ast}=\frac{\alpha_Q}{c^\ast}$, where $c^\ast$ is characterized as follows: $$c^\ast=\inf\bigg\{ c^{\frac{1}{Q-1}}:~\displaystyle\int_{\mathbb{H}^N} |D_H u^\ast|^Q~\mathrm{d}\xi\leq c \displaystyle\int_{\mathbb{H}^N} |D_H u|^Q~\mathrm{d}\xi,~u\in HW^{1,Q}(\mathbb{H}^N)\bigg\},$$
 with $u^\ast$ is regarded as the decreasing rearrangement of $u$, then for $0<\alpha\leq \alpha^\ast$ and $\frac{\alpha}{\alpha^\ast}+\frac{\beta}{Q}\leq 1$, there holds
 \begin{align}\label{eq2.5}
 \sup_{\|u\|_{HW^{1,Q}}\leq 1} ~\int_{\mathbb{H}^N}\cfrac{\Phi\big(\alpha|u|^{Q^\prime}\big)}{r(\xi)^\beta}~\mathrm{d}\xi<+\infty.    
 \end{align}
  Consequently, if $\frac{\alpha}{\alpha^\ast}+\frac{\beta}{Q}> 1$, then the integral in \eqref{eq2.5} is still finite for any $u\in HW^{1,Q}(\mathbb{H}^N)$, but the supremum is infinite if further $\frac{\alpha}{\alpha_Q}+\frac{\beta}{Q}> 1$.
\end{theorem}
Moreover, N. Lama et al. in \cite{MR2980499} established the Trudinger-Moser type inequality in $\mathbb{H}^N$ with a singular potential. Their result can be read as follows:
\begin{theorem}(\textbf{Trudinger–Moser inequality in $\mathbb{H}^N$-II })\label{thm2.7}
 Let $\tau$ be any positive real number. If for any pair $\beta,\alpha$ satisfying $0\leq\beta<Q$ and $0<\alpha\leq\alpha_Q\big(1-\frac{\beta}{Q}\big)$, where $\alpha_Q=Q\big(2\pi^N\Gamma(\frac{1}{2})\Gamma(\frac{Q-1}{2}){\Gamma(\frac{Q}{2})}^{-1}{\Gamma(N)}^{-1}\big)^{Q^\prime-1}$, then there holds
  \begin{align}\label{eq2.6}
 \sup_{\|u\|_{HW_\tau^{1,Q}}\leq 1} ~\int_{\mathbb{H}^N}\cfrac{\Phi\big(\alpha|u|^{Q^\prime}\big)}{r(\xi)^\beta}~\mathrm{d}\xi<+\infty, ~~\text{where}~\|u\|_{HW_\tau^{1,Q}}=(\tau\|u\|^Q_Q+\|D_H u\|^Q_Q)^{\frac{1}{Q}}    
 \end{align}
and $\Phi$ is defined as in $(f1)$. Consequently, if $\alpha>\alpha_Q\big(1-\frac{\beta}{Q}\big)$, then the integral in \eqref{eq2.6} is still finite for any $u\in HW^{1,Q}(\mathbb{H}^N)$, but the supremum is infinite.
\end{theorem}
\begin{remark}
    To study the problem \eqref{main problem}, we use Theorem \ref{thm2.7} with $\beta=0$ and $\tau=1$.
\end{remark}
Now we state the doubly weighted Hardy–Littlewood–Sobolev inequality in the Heisenberg group $\mathbb{H}^N$, which was proved by X. Han et al. in \cite{MR2921990}.

\begin{theorem}(\textbf{Doubly Weighted Hardy-Littlewood-Sobolev inequality in $\mathbb{H}^N$ })\label{thm2.9}
Let $1<\textbf{r},~\textbf{s}<\infty$,~$0<\lambda<Q=2N+2$ and $\alpha+\beta\geq 0$ such that $\lambda+\alpha+\beta\leq Q$,~$\alpha<\frac{Q}{\textbf{r}^\prime}$, $\beta<\frac{Q}{\textbf{s}^\prime}$, ( where $\frac{1}{\zeta}+\frac{1}{\zeta^\prime}=1~\text{for }~\zeta=\textbf{r},\textbf{s}$) and $\frac{1}{\textbf{r}}+\frac{1}{\textbf{s}}+\frac{\lambda+\alpha+\beta}{Q}=2$, then there exists a positive constant $C_{\alpha,\beta,\textbf{r},\lambda,N}$ independent of the function $f\in L^\textbf{r}(\mathbb{H}^N)$ and $g\in L^\textbf{s}(\mathbb{H}^N)$ such that 
$$\bigg|\displaystyle\int_{\mathbb{H}^N}\displaystyle\int_{\mathbb{H}^N}\cfrac{f(\xi)g(\xi^\prime)}{r(\xi)^{\alpha }~{d_K(\xi,\xi^\prime)}^\lambda r(\xi^\prime)^{\beta}}~\mathrm{d}\xi\mathrm{d}\xi^\prime\bigg|\leq C_{\alpha,\beta,\textbf{r},\lambda,N}\|f\|_{\textbf{r}}\|g\|_{\textbf{s}} . $$
Further, let 
$$ S(g(\xi))=\displaystyle \int_{\mathbb{H}^N} \frac{g(\xi^\prime)}{r(\xi)^{\alpha }~{d_K(\xi,\xi^\prime)}^\lambda r(\xi^\prime)^{\beta}}~\mathrm{d}\xi^\prime, $$
then there exists a constant $C_{\alpha,\beta,\textbf{s},\lambda,N}>0$ independent of $g\in L^\textbf{s}(\mathbb{H}^N)$ such that for any $t$ satisfying the relations $\alpha<\frac{Q}{t}$ and $1+\frac{1}{t}=\frac{1}{\textbf{s}}+\frac{\lambda+\alpha+\beta}{Q}$, we have
$$\|S(g)\|_t\leq C_{\alpha,\beta,\textbf{s},\lambda,N} \|g\|_{\textbf{s}}. $$
\end{theorem}
\section{Existence of the first solution}\label{sec3}
In this section, we will prove Theorem \ref{thm1.1}. Going forward, it will be assumed, without mentioning that the structural assumptions stated in Theorem \ref{thm1.1} are satisfied. Evidently, \eqref{main problem} has a variational structure. Indeed, since we are interested to study the nonnegative solutions of \eqref{main problem}, we define the Euler–Lagrange functional $J:\mathbf{X}\to\mathbb{R}$ associated with \eqref{main problem} by
\begin{align}\label{eq3.1}
 J(u)=\frac{1}{p}\|u\|^p_{HW^{1,p}}+\frac{1}{Q}\|u\|^Q_{HW^{1,Q}}-\frac{\mu}{s}\int_{\mathbb{H}^N}g(\xi)(u^+)^s~\mathrm{d}\xi-\frac{\gamma}{2}\int_{\mathbb{H}^N}\Bigg(\displaystyle\int_{\mathbb{H}^N}\frac{F(\eta,u)}{r(\eta)^\beta {d_K(\xi,\eta)}^\lambda}~\mathrm{d}\eta\Bigg)\frac{F(\xi,u)}{r(\xi)^\beta}~\mathrm{d}\xi.   
\end{align}
Notice that under the assumptions $(f1)$, one can easily verify that for any $\varepsilon>0$, $\nu\geq Q$, there exists positive constant $\widetilde{C}_\varepsilon=\widetilde{C}_\varepsilon(\nu)$ such that for all $(\xi,u)\in \mathbb{H}^N\times \mathbb{R}$, we have
\begin{align}\label{eq3.2}
\begin{cases}
 |f(\xi,u)|\leq \varepsilon |u|^{Q-1}+\widetilde{C}_\varepsilon |u|^{\nu-1}\Phi(\alpha_0 |u|^{Q^\prime});\\
 |F(\xi,u)|\leq \varepsilon |u|^{Q}+\widetilde{C}_\varepsilon |u|^\nu\Phi(\alpha_0 |u|^{Q^\prime}).    
\end{cases}
\end{align}
Consequently, in view of Corollary \ref{cor2.3}, we can see that for any $u\in \mathbf{X}$, there holds
\begin{align}\label{eq3.3}
F(\xi,u)\in L^\eta(\mathbb{H}^N)~~\text{for any}~~\eta\geq1.    
\end{align} Now since \eqref{eq3.3} holds, therefore by Theorem \ref{thm2.9} with $\textbf{r}=\textbf{s}$ and $\alpha=\beta$, we conclude that
\begin{align}\label{eq3.4}    \Bigg|\int_{\mathbb{H}^N}\Bigg(\displaystyle\int_{\mathbb{H}^N}\frac{F(\eta,u)}{r(\eta)^\beta {d_K(\xi,\eta)}^\lambda}~\mathrm{d}\eta\Bigg)\frac{F(\xi,u)}{r(\xi)^\beta}~\mathrm{d}\xi\Bigg|\leq C(\beta,\lambda,N)\|F(\cdot,u)\|^2_{\frac{2Q}{2Q-2\beta-\lambda}}.
\end{align}
The above inequality together with Theorem \ref{thm2.7} ensures that $J$ is well-defined, of class $ \mathcal{C}^1(\mathbb{H}^N,\mathbb{R})$ and its Gâteaux derivative is defined by 
\begin{align}\label{eq3.5}   \langle{J^\prime(u),v}\rangle=\big\langle{u,v}\big\rangle_{H,p}+\big\langle{u,v}\big\rangle_{H,Q}-\mu\displaystyle\int_{\mathbb{H}^N} g(\xi)(u^+)^{s-1}v~\mathrm{d}\xi-\gamma \displaystyle\int_{\mathbb{H}^N}\Bigg(\displaystyle\int_{\mathbb{H}^N}\frac{F(\eta,u)}{r(\eta)^\beta {d_K(\xi,\eta)}^\lambda}~\mathrm{d}\eta\Bigg)\frac{f(\xi,u)}{r(\xi)^\beta} v~\mathrm{d}\xi,
\end{align}
 for all $u,v\in\mathbf{X}$, where $\big\langle{u,v}\big\rangle_{H,t}$ is defined as in \eqref{eq1.1} for $t\in\{p,Q\}$ and $\langle{\cdot,\cdot}\rangle$ denotes the duality order pair between the dual of $\mathbf{X}$ $(~\text{denoted by}~\mathbf{X}^\ast)$  and $\mathbf{X}$ . Note that critical points of $J$ are exactly weak solutions of \eqref{main problem}. Further, the following lemma shows that every nontrivial weak solution of \eqref{main problem} is nonnegative.
 \begin{lemma}\label{lem3.1}
     The nontrivial critical points of the energy functional $J$ are nonnegative.
 \end{lemma}
\begin{proof}
    Let $u\in\mathbf{X}\setminus\{0\}$ be a critical point of the energy functional $J$. Set $u=u^+-u^-$ and choose $v=u^-$ as a test function in $\mathbf{X}$, then from \eqref{eq3.5}, we obtain
    $$\|u^-\|^p_{HW^{1,p}}+\|u^-\|^Q_{HW^{1,Q}}=0. $$
    This shows that $\|u^-\|_{HW^{1,p}}=\|u^-\|_{HW^{1,Q}}=0$ and hence $\|u^-\|_{\mathbf{X}}=0$. Therefore, we infer that $u^-=0$ a.e. in $\mathbb{H}^N$ and consequently, we have $u=u^+\geq 0$  a.e. in $\mathbb{H}^N$. This completes the proof.
\end{proof}
The following two lemma show that the functional $J$ satisfies the well-known mountain pass geometrical structures.
\begin{lemma}\label{lem3.2}
      There exist $\rho\in(0,1]$ and two positive constants $\mu^\ast$ and $\jmath$ depending upon $\rho$ such that $J(u)\geq \jmath$ for all $u\in \mathbf{X}$ with $\|u\|_{\mathbf{X}}=\rho$ and all $\mu\in(0,\mu^\ast]$ and $\gamma>0$.  
\end{lemma}
\begin{proof}
   Choose $\nu=Q+1$, then from \eqref{eq3.1} and using the fact that $\alpha>\alpha_0$, we obtain from Corollary \ref{cor2.5} that 
    \begin{align}\label{eq3.6}
       |F(\xi,u)|\leq \varepsilon |u|^{Q}+\widetilde{C}_\varepsilon |u|^{Q+1}\Phi(\alpha |u|^{Q^\prime}),~~\forall~(\xi,u)\in  \mathbb{H}^N\times \mathbb{R}.
    \end{align}
    Let $\delta\in(0,1]$ be sufficiently small enough such that $0<\|u\|_{\mathbf{X}}\leq \delta$. Then using \eqref{eq3.4}, \eqref{eq3.6}, the H\"older's inequality, Corollary \ref{cor2.3} and Corollary \ref{cor2.5}, we obtain  
\begin{align*}
\Bigg|\int_{\mathbb{H}^N}\Bigg(\displaystyle\int_{\mathbb{H}^N}\frac{F(\eta,u)}{r(\eta)^\beta {d_K(\xi,\eta)}^\lambda}~\mathrm{d}\eta\Bigg)\frac{F(\xi,u)}{r(\xi)^\beta}~\mathrm{d}\xi\Bigg|&\leq  C\Bigg[\displaystyle\int_{\mathbb{H}^N} \bigg(\varepsilon |u|^{Q}+\widetilde{C}_\varepsilon |u|^{Q+1}\Phi(\alpha |u|^{Q^\prime})\bigg)^{\frac{2Q}{2Q-2\beta-\lambda}}~\mathrm{d}\xi\Bigg]^{\frac{2Q-2\beta-\lambda}{Q}} \notag\\
& \leq C\Bigg[\|u\|^{2Q}_{\frac{2Q^2}{2Q-2\beta-\lambda}}+\Bigg(\displaystyle\int_{\mathbb{H}^N} |u|^{\frac{2Q(Q+1)}{2Q-2\beta-\lambda}}\Phi\bigg(\frac{2Q\alpha |u|^{Q^\prime} }{2Q-2\beta-\lambda} \bigg)~\mathrm{d}\xi\Bigg)^{\frac{2Q-2\beta-\lambda}{Q}}\Bigg] \notag\\
& \leq C\Bigg[ \|u\|^{2Q}_{\frac{2Q^2}{2Q-2\beta-\lambda}}+\|u\|^{2(Q+1)}_{\frac{4Q(Q+1)}{2Q-2\beta-\lambda}} \Bigg(\displaystyle\int_{\mathbb{H}^N}\Phi\bigg(\frac{4Q\alpha \|u\|_{\mathbf{X}}^{Q^\prime}|\widetilde{u}|^{Q^\prime} }{2Q-2\beta-\lambda} \bigg)\mathrm{d}\xi\Bigg)^{\frac{2Q-2\beta-\lambda}{2Q}}\Bigg],
\end{align*}
where $\widetilde{u}=u/\|u\|_{\mathbf{X}}$ and hence one can notice that $\|\widetilde{u}\|_{HW^{1,Q}}\leq \|\widetilde{u}\|_{\mathbf{X}}=1$. Since $\|u\|_{\mathbf{X}}$ is very small enough, without loss of generality we can choose $\frac{4Q\alpha \|u\|_{\mathbf{X}}^{Q^\prime} }{2Q-2\beta-\lambda}\leq \alpha_Q$. Now by applying Theorem \ref{thm2.7}, we deduce that
\begin{align}\label{eq3.7}
\Bigg|\int_{\mathbb{H}^N}\Bigg(\displaystyle\int_{\mathbb{H}^N}\frac{F(\eta,u)}{r(\eta)^\beta {d_K(\xi,\eta)}^\lambda}~\mathrm{d}\eta\Bigg)\frac{F(\xi,u)}{r(\xi)^\beta}~\mathrm{d}\xi\Bigg|&\leq  C\Bigg[ \|u\|^{2Q}_{\frac{2Q^2}{2Q-2\beta-\lambda}}+\|u\|^{2(Q+1)}_{\frac{4Q(Q+1)}{2Q-2\beta-\lambda}} \Bigg] \notag\\
& \leq C \big(\|u\|^{2Q}_{\mathbf{X}}+ \|u\|^{2(Q+1)}_{\mathbf{X}}\big)\leq C\|u\|^{2Q}_{\mathbf{X}},
\end{align}
where the last step in the above inequality is obtained by using $0<\|u\|_{\mathbf{X}}\leq 1$ and also we used that the constant $C>0$ in the above estimates varies from step to step.\\
Consequently, from \eqref{eq3.1}, \eqref{eq3.7} along with the H\"older's inequality and Corollary \ref{cor2.3}, it follows that
\begin{align}\label{eq3.8}
  J(u)&\geq \frac{1}{2^{Q-1}Q}\|u\|^{Q}_{\mathbf{X}}-\frac{\mu}{s} \|g\|_{\vartheta}\|u\|^s_{\mathbf{X}}- C\|u\|^{2Q}_{\mathbf{X}}.
\end{align}
Define $$h(\ell)=\frac{\ell^Q}{2^Q Q}-C\ell^{2Q}~\text{for all}~\ell\in(0,\delta].$$ Then one can observe that $h$ admits a positive maximum $\jmath$ on $(0,\delta]$ at a point $\rho\in(0,\delta]$. Moreover, for all $u\in\mathbf{X}$ with $\|u\|_{\mathbf{X}}=\rho$, we obtain from \eqref{eq3.8} that
$$J(u)\geq \frac{\rho^Q}{2^{Q-1}Q} -\frac{\mu}{s} \|g\|_{\vartheta} ~\rho^s- C\rho^{2Q}\geq h(\rho)=\jmath>0~~\text{for all}~ \mu\in(0,\mu^\ast]~\text{with}~ \mu^\ast=\frac{s\rho^{Q-s}}{2^Q Q\|g\|_{\vartheta}}. $$
This finishes the proof.   
\end{proof}
\begin{lemma}\label{lem3.3}
    There exists a nonnegative function $e\in \mathbf{X}$, independent of $\mu$ such that $\|e\|_{\mathbf{X}}>\rho$ and $J(e)<0$ for all $\mu>0$ and for all $\gamma>0$.
\end{lemma}
\begin{proof}
    Choose $u_0\in \mathbf{X}\setminus\{0\},~u_0\geq 0 $ and define a function $\mathcal{H}(t)=\Psi\bigg(\displaystyle\frac{tu_0}{\|u_0\|_{\mathbf{X}}}\bigg)$ for any $t>0$, where $\Psi$ is defined by
    $$\Psi(u)=\frac{1}{2}\int_{\mathbb{H}^N}\Bigg(\displaystyle\int_{\mathbb{H}^N}\frac{F(\eta,u)}{r(\eta)^\beta {d_K(\xi,\eta)}^\lambda}~\mathrm{d}\eta\Bigg)\frac{F(\xi,u)}{r(\xi)^\beta}~\mathrm{d}\xi. $$
    Using $(f2)$, we have the following estimates for any $t>0$
    \begin{align}\label{eq3.9}
      \mathcal{H}^\prime(t)&= \frac{1}{t}\Bigg\langle{\Psi^\prime\bigg(\displaystyle\frac{tu_0}{\|u_0\|_{\mathbf{X}}}\bigg),\displaystyle\frac{tu_0}{\|u_0\|_{\mathbf{X}}}}\Bigg\rangle=\frac{\sigma}{t} \displaystyle\int_{\mathbb{H}^N}\Bigg(\displaystyle\int_{\mathbb{H}^N}\frac{F\Big(\eta,\frac{tu_0}{\|u_0\|_{\mathbf{X}}}\Big)}{r(\eta)^\beta {d_K(\xi,\eta)}^\lambda}~\mathrm{d}\eta\Bigg)\frac{f\Big(\xi,\frac{tu_0}{\|u_0\|_{\mathbf{X}}}\Big)}{r(\xi)^\beta} \frac{tu_0}{\sigma\|u_0\|_{\mathbf{X}}}~\mathrm{d}\xi \notag\\
      &\geq \frac{\sigma}{2t} \displaystyle\int_{\mathbb{H}^N}\Bigg(\displaystyle\int_{\mathbb{H}^N}\frac{F\Big(\eta,\frac{tu_0}{\|u_0\|_{\mathbf{X}}}\Big)}{r(\eta)^\beta {d_K(\xi,\eta)}^\lambda}~\mathrm{d}\eta\Bigg)\frac{F\Big(\xi,\frac{tu_0}{\|u_0\|_{\mathbf{X}}}\Big)}{r(\xi)^\beta} ~\mathrm{d}\xi=\frac{\sigma}{t}\mathcal{H}(t).
    \end{align}
Integrating \eqref{eq3.9} on $[1,s\|u_0\|_{\mathbf{X}}]$ with  $s>\frac{1}{\|u_0\|_{\mathbf{X}}}$, we obtain $\mathcal{H}(s\|u_0\|_{\mathbf{X}})\geq \mathcal{H}(1)s^\sigma \|u_0\|^\sigma_{\mathbf{X}}$, which implies at once that 
\begin{align}\label{eq3.10}
    \Psi(su_0)\geq \Psi\bigg(\displaystyle\frac{u_0}{\|u_0\|_{\mathbf{X}}}\bigg) s^\sigma\|u_0\|^\sigma_{\mathbf{X}}.
\end{align}
Hence, it follows from \eqref{eq3.1} and \eqref{eq3.10} that
$$J(su_0)\leq \frac{s^p}{p}\|u_0\|^p_{HW^{1,p}}+\frac{s^Q}{Q}\|u_0\|^Q_{HW^{1,Q}}-\frac{\gamma s^\sigma \|u_0\|^\sigma_{\mathbf{X}}}{2}~ \Psi\bigg(\displaystyle\frac{u_0}{\|u_0\|_{\mathbf{X}}}\bigg)\to -\infty~\text{as}~s\to\infty,$$
since $p<Q<\sigma$. Thus, we can choose $s>\frac{1}{\|u_0\|_{\mathbf{X}}}$ large enough such that $e=su_0$ with $J(e)<0$ and $\|e\|_{\mathbf{X}}>\rho$. This completes the proof.   
\end{proof}
\begin{lemma}\label{lem3.4}
    Let $\rho\in(0,1]$ as in Lemma \ref{lem3.3}, then for all $\mu\in(0,\mu^\ast]$ as in Lemma \ref{lem3.3} and for all $\gamma>0$, there exist a sequence $\{u_n\}_n$ of nonnegative functions and some nonnegative function $u_{\mu,\gamma}$ in $\overline{B}_\rho$ such that for all $n\in \mathbb{N}$,
    \begin{equation}\label{eq3.11}
    \begin{split}
     \|u_n\|_{\mathbf{X}}<\rho,~m_{\mu,\gamma}\leq J(u_n)\leq m_{\mu,\gamma}+\frac{1}{n},~u_n\rightharpoonup u_{\mu,\gamma}~\text{in}~\mathbf{X},\\
     u_n\to u_{\mu,\gamma}~\text{a.e. in}~\mathbb{H}^N~\text{and}~J^\prime(u_n)\to 0~\text{in}~\mathbf{X}^\ast~\text{as}~n\to\infty,\\
    \end{split}
\end{equation}
where $$ m_{\mu,\gamma}=\inf\{J(u):u\in\overline{B}_\rho\}<0.$$    
\end{lemma}
\begin{proof}
Let $v\in\mathbf{X}\setminus\{0\}$, $v\geq 0$, with $\|v\|_{\mathbf{X}}=1$. Now using $(f3)$, we infer that $F(\xi,t)\geq Ct^{\frac{\sigma}{2}}$ for some constant $C>0$, for a.e. $\xi\in \mathbb{H}^N$ and for all $t>0$. Hence, for all $\tau\in (0,1]$ sufficiently small enough, we have 
\begin{align*}J(\tau v)&\leq \frac{2\tau^{p}}{p}-\frac{\mu\tau^s}{s}\|v\|^s_{s,g}-\frac{C^2\gamma\tau^\sigma}{2}\displaystyle\int_{\mathbb{H}^N}\Bigg(\displaystyle\int_{\mathbb{H}^N}\frac{|v(\eta)|^{\frac{\sigma}{2}}}{r(\eta)^\beta {d_K(\xi,\eta)}^\lambda}~\mathrm{d}\eta\Bigg)\frac{|v(\xi)|^{\frac{\sigma}{2}}}{r(\xi)^\beta}~\mathrm{d}\xi\\
&\leq \frac{2\tau^{p}}{p}-\frac{\mu\tau^s}{s}\|v\|^s_{s,g} <0,\end{align*}
because of $s<p$. From the above estimates and \eqref{eq3.8}, we obtain
$$-\infty<m_{\mu,\gamma}=\inf _{\|u\|\leq \rho}J(u)\leq \inf _{\tau\in(0,\rho]}J(\tau v)<0. $$
Consequently, we deduce that
\begin{equation}\label{eq3.12}
  -\infty< m_{\mu,\gamma}=\inf\{J_{\mu,\gamma}(u):u\in\overline{B}_\rho\}<0. 
\end{equation}
It follows that the functional $J$ is bounded from below and of class $\mathcal{C}^1$ on $\overline{B}_\rho$. Further, we know that $\overline{B}_\rho$ is a complete metric space with the metric given by the norm of $\mathbf{X}$. In light of Lemma \ref{lem3.2}, we infer that
\begin{equation}\label{eq3.13}
 \inf _{\partial B_\rho}J(u)\geq \jmath>0 . 
\end{equation}
In view of \eqref{eq3.12} and \eqref{eq3.13}, for $n$ large enough, we can choose
\begin{equation}\label{eq3.14}
    \frac{1}{n}\in \big(0, \inf _{\partial B_\rho}J(u)-\inf _{\overline{B}_\rho }J(u)\big).
\end{equation}
Due to Ekeland variational principle for the functional $J:\overline{B}_\rho\to\mathbb{R}$, we can find a sequence $\{u_n\}_n\subset \overline{B}_\rho $ such that
\begin{equation}\label{eq3.15}
  m_{\mu,\gamma}\leq J(u_n)\leq m_{\mu,\gamma}+\frac{1}{n}~\text{and} ~J(u_n)\leq J(u)+\frac{1}{n}\|u_n-u\|_{\mathbf{X}}, 
\end{equation}
for all $u\in\overline{B}_\rho$, with $u\neq u_n$ for each $n\in\mathbb{N}$. Now, it follows from \eqref{eq3.14} and \eqref{eq3.15} that
$$ J(u_n)\leq m_{\mu,\gamma}+\frac{1}{n}=\inf _{\overline{B}_\rho }J(u)+\frac{1}{n}<\inf _{\partial B_\rho}J(u).$$
From the above inequality, we deduce that $\{u_n\}_n\subset B_\rho $, that is, $\|u_n\|_{\mathbf{X}}<\rho$ for all $n\in\mathbb{N}$. Let for all $\varphi\in B_1$, we choose $t>0$ sufficiently small such that $u_n+t\varphi\in\overline{B}_\rho $ holds, for each $n\in\mathbb{N}$. Consequently, from \eqref{eq3.15}, we obtain
$$\langle{J^\prime(u_n),\varphi}\rangle=\lim_{t\to 0^+}\frac{J(u_n+t\varphi)-J(u_n)}{t}\geq -\frac{1}{n},~\forall~\varphi\in B_1. $$
Since $\varphi\in B_1$ is arbitrary, we conclude that $|\langle{J^\prime(u_n),\varphi}\rangle|\leq \frac{1}{n}$ for all $\varphi\in B_1$. This together with \eqref{eq3.15}, we have that
\begin{equation}\label{eq3.16}
  J(u_n) \to m_{\mu,\gamma}~\text{and}~ J^\prime(u_n)\to 0~\text{in}~\mathbf{X}^\ast~\text{as}~n\to\infty.
\end{equation}
In addition, since $\{u_n\}_n$ is bounded, there exists a subsequence still denoted by the same symbol and $u_{\mu,\gamma},u_1,u_2\in \overline{B}_\rho$  such that $u_n\rightharpoonup u_{\mu,\gamma}$, $u_n^+\rightharpoonup u_1$ and $u_n^-\rightharpoonup u_2$ in $\mathbf{X}$ as $n\to\infty$. By Corollary \ref{cor2.3}, we obtain for any $s\in[1,\infty)$ that $u_n\to u_{\mu,\gamma}$, $u_n^+\to u_1$ and $u_n^-\to u_2$ in $L^s(B_R)$ as $n\to\infty$ for any $R>0$. Using the fact that the maps $u\mapsto u^{\pm}$ are continuous from $L^s(B_R)$ into itself, we infer that $u_1=u_{\mu,\gamma}^+$ and $u_2=u_{\mu,\gamma}^-$. This shows that $u_n\to u_{\mu,\gamma}$, $u_n^+\to u_1^+$ and $u_n^-\to u_2^-$ a.e. in $\mathbb{H}^N$ as $n\to\infty$. Further, since $\langle{J^\prime(u_n),u_n^-}\rangle\to 0$ as $n\to\infty$, therefore we obtained from \eqref{eq3.5} that 
 $$\|u_n^-\|^p_{HW^{1,p}}+\|u_n^-\|^Q_{HW^{1,Q}}=o_n(1)~~\text{as}~n\to\infty. $$
Consequently, we deduce that $u_n^-\to 0$ in $\mathbf{X}$ as $n\to\infty$ and thus $u_n^-\to 0$ a.e. in $\mathbb{H}^N$ as $n\to\infty$. Therefore, we infer that $u_{\mu,\gamma}^-=0$ a.e. in $\mathbb{H}^N$ and $u_{\mu,\gamma}=u_{\mu,\gamma}^+\geq 0$ a.e. in $\mathbb{H}^N$. Hence, without loss of generality, we can assume $u_n=u_n^+$, since $u_n^-\to 0$ in $\mathbf{X}$ as $n\to\infty$. It follows that $\{u_n\}_n$ and $u_{\mu,\gamma}$ are nonnegative. This completes the proof.
\end{proof}
\begin{proposition}\label{prop3.5}
Let $\{u_n\}_n\subset\mathbf{X}$ be a sequence and $u\in \mathbf{X}$ be such that $u_n\rightharpoonup u$  in $\mathbf{X}$ as $n\to\infty$. Furthermore, if there holds
\begin{equation}\label{eq3.17}
   \limsup_{n\to\infty}\|u_n\|^{Q^\prime}_{HW^{1,Q}}<\alpha_Q, 
\end{equation}
then for any $v\in\mathcal{C}^\infty_c(\mathbb{H}^N)$, we have up to a subsequence,
\begin{align*}
    \lim_{n\to\infty} \displaystyle\int_{\mathbb{H}^N}\Bigg(\displaystyle\int_{\mathbb{H}^N}\frac{F(\eta,u_n)}{r(\eta)^\beta {d_K(\xi,\eta)}^\lambda}~\mathrm{d}\eta\Bigg)\frac{f(\xi,u_n)}{r(\xi)^\beta} v~\mathrm{d}\xi=\displaystyle\int_{\mathbb{H}^N}\Bigg(\displaystyle\int_{\mathbb{H}^N}\frac{F(\eta,u)}{r(\eta)^\beta {d_K(\xi,\eta)}^\lambda}~\mathrm{d}\eta\Bigg)\frac{f(\xi,u)}{r(\xi)^\beta} v~\mathrm{d}\xi.
\end{align*}
\end{proposition}
\begin{proof}
By the hypothesis and Corollary \ref{cor2.3}, we obtain that  $\{u_n\}$ is bounded sequence in $\mathbf{X}$ and $u_n\to u~\text{a.e. in }~\mathbb{H}^N$ as $n\to\infty$. Consequently, by using the growth conditions on $F$ (see \eqref{eq3.2}) together with \eqref{eq3.17} and Theorem \ref{thm2.7} imply that  $\{F(\cdot,u_n)\}_n$ is bounded in $L^{\frac{2Q}{2Q-2\beta-\lambda}}(\mathbb{H}^N)$.  
In addition, due to the continuity of $F$, we deduce that $F(\xi,u_n)\to F(\xi,u) $  a.e. in $\mathbb{H}^N~\text{as}~n\to\infty$. It follows at once that
$$F(\xi,u_n)\rightharpoonup F(\xi,u)~\text{in}~L^{\frac{2Q}{2Q-2\beta-\lambda}}(\mathbb{H}^N)~\text{as}~n\to\infty.$$
By Theorem \ref{thm2.9}, we known that the application
$$L^{\frac{2Q}{2Q-2\beta-\lambda}}(\mathbb{H}^N)\ni h(\xi)\mapsto \displaystyle \int_{\mathbb{H}^N} \frac{h(\eta)}{r(\xi)^{\beta }~{d_K(\xi,\eta)}^\lambda r(\eta)^{\beta}}~\mathrm{d}\eta\in L^{\frac{2Q}{2\beta+\lambda}}(\mathbb{H}^N) $$
is a linear and bounded operator from $L^{\frac{2Q}{2Q-2\beta-\lambda}}(\mathbb{H}^N)$ into $L^{\frac{2Q}{2\beta+\lambda}}(\mathbb{H}^N)$. Hence, we deduce that
\begin{align}\label{eq3.19}
  \displaystyle \int_{\mathbb{H}^N} \frac{F(\eta,u_n)}{r(\xi)^{\beta }~{d_K(\xi,\eta)}^\lambda r(\eta)^{\beta}}~\mathrm{d}\eta\rightharpoonup \displaystyle \int_{\mathbb{H}^N} \frac{F(\eta,u)}{r(\xi)^{\beta }~{d_K(\xi,\eta)}^\lambda r(\eta)^{\beta}}~\mathrm{d}\eta~\text{in}~L^{\frac{2Q}{2\beta+\lambda}}(\mathbb{H}^N)~\text{as}~n\to\infty.
\end{align}
In light of \eqref{eq3.19}, there exists a constant $\widetilde{C}>0$ such that
\begin{equation}\label{eq3.20}
    \Bigg\| \displaystyle \int_{\mathbb{H}^N} \frac{F(\eta,u_n)}{r(\xi)^{\beta }~{d_K(\xi,\eta)}^\lambda r(\eta)^{\beta}}~\mathrm{d}\eta\Bigg\|_{\frac{2Q}{2\beta+\lambda}}\leq \widetilde{C}.
\end{equation}
Also, thanks to \eqref{eq3.19}, for every $\varphi\in L^{\frac{2Q}{2Q-2\beta-\lambda}}(\mathbb{H}^N)$, we get
\begin{align*}
    \displaystyle\int_{\mathbb{H}^N}\Bigg(\displaystyle\int_{\mathbb{H}^N}\frac{F(\eta,u_n)}{r(\eta)^\beta {d_K(\xi,\eta)}^\lambda}~\mathrm{d}\eta\Bigg)\frac{\varphi(\xi)}{r(\xi)^\beta} ~\mathrm{d}\xi\to\displaystyle\int_{\mathbb{H}^N}\Bigg(\displaystyle\int_{\mathbb{H}^N}\frac{F(\eta,u)}{r(\eta)^\beta {d_K(\xi,\eta)}^\lambda}~\mathrm{d}\eta\Bigg)\frac{\varphi(\xi)}{r(\xi)^\beta} ~\mathrm{d}\xi~\text{as}~n\to\infty.   
\end{align*}
In particular, we have for all $v\in\mathcal{C}^\infty_c(\mathbb{H}^N)$, there holds
\begin{align}\label{eq3.21}
\displaystyle\int_{\mathbb{H}^N}\Bigg(\displaystyle\int_{\mathbb{H}^N}\frac{F(\eta,u_n)}{r(\eta)^\beta {d_K(\xi,\eta)}^\lambda}~\mathrm{d}\eta\Bigg)\frac{f(\xi,u)}{r(\xi)^\beta} v~\mathrm{d}\xi\to\displaystyle\int_{\mathbb{H}^N}\Bigg(\displaystyle\int_{\mathbb{H}^N}\frac{F(\eta,u)}{r(\eta)^\beta {d_K(\xi,\eta)}^\lambda}~\mathrm{d}\eta\Bigg)\frac{f(\xi,u)}{r(\xi)^\beta} v~\mathrm{d}\xi~\text{as}~n\to\infty.      
\end{align}
Now, we claim that for every $v\in\mathcal{C}^\infty_c(\mathbb{H}^N)$, there holds 
\begin{align}\label{eq3.22}
\displaystyle\int_{\mathbb{H}^N}\Bigg(\displaystyle\int_{\mathbb{H}^N}\frac{F(\eta,u_n)}{r(\eta)^\beta {d_K(\xi,\eta)}^\lambda}~\mathrm{d}\eta\Bigg)\frac{(f(\xi,u_n)-f(\xi,u))}{r(\xi)^\beta} v~\mathrm{d}\xi\to 0~\text{as}~n\to\infty.      
\end{align}
Observe that, to check the validity of \eqref{eq3.22}, it is enough to show:
\begin{align}\label{eq3.23}
    \|(f(\cdot,u_n)-f(\cdot,u)) v\|_{{\frac{2Q}{2Q-2\beta-\lambda},~ \text{supp}(v)}}\to 0~\text{as}~n\to\infty,
\end{align}
since we can use the H\"older's inequality and \eqref{eq3.20}.
Notice that $f(\cdot,u_n)v\to f(\cdot,u)v$  a.e. in $\mathbb{H}^N~\text{as}~n\to\infty$. Further, it follows from \eqref{eq3.17} that there exists $n_0\in\mathbb{N}$ and $m>0$ such that $\|u_n\|^{Q^\prime}_{HW^{1,Q}}<m<\alpha_Q$ for all $n\geq n_0$. Let $r$, $r^\prime>1$ satisfying the relation $\frac{1}{r}+\frac{1}{r^\prime}=1$. Now, choose $\alpha>\alpha_0$ very close to $\alpha_0$ and $r^\prime>1$ very close to 1 in such a way that we still have $\displaystyle\frac{2\alpha Q r^\prime }{2Q-2\beta-\lambda} \|u_n\|^{Q^\prime}_{HW^{1, Q}}<m<\alpha_Q$ for all $n\geq n_0$.  Moreover, we obtain the following estimates for all $n\geq n_0$ by employing $(f1)$, Corollary \ref{cor2.3}, Corollary \ref{cor2.5}, and the H\"older's inequality:

\begin{align}\label{eq3.24}
  \int_{\text{supp}(v)}|f(\xi,u_n)v|^{\frac{2Q}{2Q-2\beta-\lambda}}~\mathrm{d}\xi&\leq  C \int_{\text{supp}(v)}|v|^{\frac{2Q}{2Q-2\beta-\lambda}} \Bigg( |u_n|^{\frac{2Q(Q-1)}{2Q-2\beta-\lambda}}+\Phi\bigg(\frac{2Q\alpha}{2Q-2\beta-\lambda} |u_n|^{Q^\prime}\bigg)\Bigg)~\mathrm{d}\xi \notag\\
    & \leq C \Bigg[\|u_n\|^{\frac{2Q(Q-1)}{2Q-2\beta-\lambda}}_{\frac{2Q^2}{2Q-2\beta-\lambda}}\bigg(\int_{\text{supp}(v)}|v|^{\frac{2Q^2}{2Q-2\beta-\lambda}}~\mathrm{d}\xi\bigg)^{\frac{1}{Q}}+\bigg(\int_{\mathbb{H}^N}\Phi\bigg(\frac{2\alpha Q r^\prime}{2Q-2\beta-\lambda} |u_n|^{Q^\prime}\bigg)~\mathrm{d}\xi\bigg)^{\frac{1}{r^\prime}} \notag\\
    & \qquad~~~\times\bigg(\int_{\text{supp}(v)}|v|^{\frac{2Qr}{2Q-2\beta-\lambda}}~\mathrm{d}\xi\bigg)^{\frac{1}{r}}\Bigg] \notag\\
    & \leq C_1\Bigg[\|u_n\|^{\frac{2Q(Q-1)}{2Q-2\beta-\lambda}}_{\mathbf{X}}\bigg(\int_{\text{supp}(v)}|v|^{\frac{2Q^2}{2Q-2\beta-\lambda}}\mathrm{d}\xi\bigg)^{\frac{1}{Q}} +\bigg(\int_{\mathbb{H}^N}\Phi\bigg(\frac{2\alpha Q r^\prime \|u_n\|^{Q^\prime}_{HW^{1, Q}}}{2Q-2\beta-\lambda}|\widetilde{u}_n|^{Q^\prime} \bigg)\mathrm{d}\xi\bigg)^{\frac{1}{r^\prime}}\notag\\
       & \qquad~~~ \times\bigg(\int_{\text{supp}(v)}|v|^{\frac{2Qr}{2Q-2\beta-\lambda}}~\mathrm{d}\xi\bigg)^{\frac{1}{r}}\Bigg] \notag\\
    &\leq C_2 \Bigg[\bigg(\int_{\text{supp}(v)}|v|^{\frac{2Q^2}{2Q-2\beta-\lambda}}~\mathrm{d}\xi\bigg)^{\frac{1}{Q}}  +\bigg(\int_{\text{supp}(v)}|v|^{\frac{2Qr}{2Q-2\beta-\lambda}}~\mathrm{d}\xi\bigg)^{\frac{1}{r}}\Bigg],  
    \end{align}
where $C, C_1$ are positive constants, $\widetilde{u}_n=u_n/{\|u_n\|_{HW^{1,Q}}}$ and $C_2$ is defined as follows
\begin{align*}
    C_2&=C_1\max\Bigg\{ \sup_{n\in\mathbb{N}}\|u_n\|^{\frac{2Q(Q-1)}{2Q-2\beta-\lambda}}_{\mathbf{X}},~\bigg(\sup_{n\geq n_0}\int_{\mathbb{H}^N}\Phi\bigg(\frac{2\alpha Q r^\prime \|u_n\|^{Q^\prime}_{HW^{1, Q}}}{2Q-2\beta-\lambda}|\widetilde{u}_n|^{Q^\prime} \bigg)~\mathrm{d}\xi\bigg)^{\frac{1}{r^\prime}}~\Bigg\},
\end{align*}
which is finite because of the uniform boundedness of the sequence $\{u_n\}_n$ in $\mathbf{X}$ and Theorem \ref{thm2.7}. Observe that the right hand side of $\eqref{eq3.24}$ is finite, therefore we conclude that $\{|f(\cdot,u_n)v|^{\frac{2Q}{2Q-2\beta-\lambda}}\}_{n\geq n_0}$ is bounded in $L^1(\text{supp}(v))$. Consequently, we can also see that the sequence $\{|f(\cdot,u_n)v|^{\frac{2Q}{2Q-2\beta-\lambda}}\}_{n\geq n_0}$ is uniformly absolutely integrable and tight over $\text{supp}(v)$. Thus, by applying the Vitali's convergence theorem, we obtain
$$\|f(\cdot,u_n)v\|_{\frac{2Q}{2Q-2\beta-\lambda},~ \text{supp}(v)}\to\|f(\cdot,u)) v\|_{{\frac{2Q}{2Q-2\beta-\lambda},~ \text{supp}(v)}}~\text{as}~n\to\infty. $$
This together with Br\'ezis-Lieb  lemma implies the validity of \eqref{eq3.23} and hence \eqref{eq3.22} holds. Therefore, combining \eqref{eq3.21} and \eqref{eq3.22}, the proof of Proposition \ref{prop3.5} is completely finished.
\end{proof}
\begin{proposition}\label{prop3.6}
Let $\{u_n\}_n\subset\mathbf{X} $ be a sequence and there exists $u\in \mathbf{X}$ such that $u_n\rightharpoonup u$  in $\mathbf{X}$ as $n\to\infty$. In addition, if $J^\prime(u_n)\to 0$ in $\mathbf{X}^\ast$ as $n\to\infty$ and \eqref{eq3.17} hold true, then $D_H u_n\to D_H u$ a.e. in $\mathbb{H}^N$ as $n\to\infty$. Consequently, for $t\in\{p,Q\}$, there holds
\begin{align}\label{eq3.333}
  |D_H u_n|_H^{t-2}D_H u_n &\rightharpoonup |D_H u|_H^{t-2}D_H u~\text{in}~L^{\frac{t}{t-1}}(\mathbb{H}^N,\mathbb{R}^{2N})~\text{and}~  |u_n|^{t-2} u_n\rightharpoonup |u|^{t-2} u~\text{in}~L^{\frac{t}{t-1}}(\mathbb{H}^N)~\text{as}~n\to\infty.
\end{align}

\end{proposition}
\begin{proof}
 By the hypotheses and Corollary \ref{cor2.3}, we obtain that
\begin{equation}\label{eq3.99}
   u_n\rightharpoonup u~\text{in}~\mathbf{X},~ u_n\to u ~\text{in}~L^s(B_R)~\text{and }~u_n\to u~\text{a.e. in } ~\mathbb{H}^N~\text{as}~n\to\infty 
\end{equation} 
for $s\in[1,\infty)$ and all $R>0$. Fix $R>0$ and $\psi\in \mathcal{C}^\infty_c(\mathbb{H}^N)$ such that $0\leq \psi\leq 1$ in $\mathbb{H}^N$, $\psi\equiv 1$ in $B_R$ and  $\psi\equiv 0$ in $B^c_{2R}$. In addition to this, since $J\in \mathcal{C}^1(\mathbf{X},\mathbb{R})$, $u_n\rightharpoonup u$ in $\mathbf{X}$, and $J^\prime(u_n)\to 0$ in $\mathbf{X}^\ast$ as $n\to\infty$, therefore we have 
\begin{equation}\label{eq3.25}
 \langle{J^\prime(u_n)-J^\prime(u),(u_n-u)\psi}\rangle=o_n(1)~\text{as}~n\to\infty.   
\end{equation}
For any $n\in\mathbb{N}$ and $t\in\{p,Q\}$, let us define 
$$ \boldsymbol{T}^t_n=\big(|D_H u_n|_H^{t-2} D_H u_n-|D_H u|_H^{t-2} D_H u, D_H u_n-D_H u\big)_H+\big(| u_n|^{t-2} u_n-|u|^{t-2}u\big) (u_n-u). $$
Moreover, by convexity, we can see that $\big(|D_H u_n|_H^{t-2} D_H u_n-|D_H u|_H^{t-2} D_H u, D_H u_n-D_H u\big)_H\geq 0~\text{a.e. in} ~\mathbb{H}^N$ and $\big(| u_n|^{t-2} u_n-|u|^{t-2}u\big) (u_n-u)\geq 0~\text{a.e. in} ~\mathbb{H}^N$ for any $n\in\mathbb{N}$ and $t\in\{p,Q\}$.  In virtue of Simon's inequality (see \cite{MR0519432}) with $Q=2N+2> 2$ and \eqref{eq3.25}, there exists $\kappa>0$ such that
\begin{align}\label{eq3.26}
 \kappa^{-1}\|u_n-u\|^Q_{HW^{1,Q}(B_R)}&= \kappa^{-1}\bigg(\int_{B_R} |D_H u_n-D_H u|_H^{Q}~\mathrm{d}\xi+ \int_{B_R} |u_n- u|^{Q}~\mathrm{d}\xi\bigg) \notag\\
 & \leq  \int_{B_R} \boldsymbol{T}^Q_n ~\mathrm{d}\xi \leq \displaystyle\sum_{t\in\{p,Q\}} \bigg(\int_{B_R}\boldsymbol{T}^t_n ~\mathrm{d}\xi\bigg) \leq \displaystyle\sum_{t\in\{p,Q\}} \bigg(\int_{\mathbb{H}^N}\boldsymbol{T}^t_n~ \psi~\mathrm{d}\xi\bigg) \notag\\
 & =-\displaystyle\sum_{t\in\{p,Q\}} \bigg(\int_{\mathbb{H}^N}\big(|D_H u_n|_H^{t-2} D_H u_n-|D_H u|_H^{t-2} D_H u, D_H \psi\big)_H (u_n-u) ~\mathrm{d}\xi\bigg) \notag\\
 & \quad +\gamma \displaystyle\int_{\mathbb{H}^N}\Bigg[\Bigg(\displaystyle\int_{\mathbb{H}^N}\frac{F(\eta,u_n)}{r(\eta)^\beta {d_K(\xi,\eta)}^\lambda}\mathrm{d}\eta\Bigg)\frac{f(\xi,u_n)}{r(\xi)^\beta} -\Bigg(\displaystyle\int_{\mathbb{H}^N}\frac{F(\eta,u)}{r(\eta)^\beta {d_K(\xi,\eta)}^\lambda}\mathrm{d}\eta\Bigg)\frac{f(\xi,u)}{r(\xi)^\beta}\Bigg](u_n-u)\psi\mathrm{d}\xi\notag\\
 & \quad +\mu\int_{\mathbb{H}^N} g(\xi)\big((u_n^+)^{s-1}-(u^+)^{s-1}\big)(u_n-u)\psi~~\mathrm{d}\xi+o_n(1)~\text{as}~n\to\infty. 
\end{align}
Due to the H\"older's inequality and \eqref{eq3.24}, we get for $t\in\{p,Q\}$ that 
\begin{align*}
    \Bigg| \int_{\mathbb{H}^N}\big(|D_H u_n|_H^{t-2} D_H u_n-|D_H u|_H^{t-2} D_H u, D_H \psi\big)_H (u_n-u)\mathrm{d}\xi\Bigg|&\leq \|D_H \psi\|_{\infty}\big(\|D_H u_n\|^{t-1}_t+\|D_H u\|^{t-1}_t\big)\bigg(\int_{B_{2R}}|u_n-u|^t\mathrm{d}\xi\bigg)^{\frac{1}{t}} \notag\\
    & \quad\to 0~\text{as}~n\to\infty.\end{align*}
It follows that for $t\in\{p,Q\}$, we have
\begin{equation}\label{eq3.27}
 \lim_{n\to\infty} \int_{\mathbb{H}^N}\big(|D_H u_n|_H^{t-2} D_H u_n-|D_H u|_H^{t-2} D_H u, D_H \psi\big)_H (u_n-u) ~\mathrm{d}\xi =0.  
\end{equation}
In the same way, once more via the H\"older's inequality, we obtain
$$\bigg|\int_{\mathbb{H}^N} g(\xi)\big((u_n^+)^{s-1}-(u^+)^{s-1}\big)(u_n-u)\psi~\mathrm{d}\xi\bigg| \leq \big(\|u_n\|^{s-1}_{s,g}+\|u\|^{s-1}_{s,g}\big)\|g\|^{\frac{1}{s}}_\vartheta \bigg(\int_{B_{2R}}|u_n-u|^Q~\mathrm{d}\xi\bigg)^{\frac{1}{Q}}\to 0~\text{as}~n\to\infty.$$
This shows that 
\begin{equation}\label{eq3.28}
 \lim_{n\to\infty}\int_{\mathbb{H}^N} g(\xi)\big((u_n^+)^{s-1}-(u^+)^{s-1}\big)(u_n-u)\psi~\mathrm{d}\xi =0.  
\end{equation}
Next, we claim that 
\begin{equation}\label{eq3.29}
   \lim_{n\to\infty} \displaystyle\int_{\mathbb{H}^N}\Bigg(\displaystyle\int_{\mathbb{H}^N}\frac{F(\eta,u_n)}{r(\eta)^\beta {d_K(\xi,\eta)}^\lambda}~\mathrm{d}\eta\Bigg)\frac{f(\xi,u_n)}{r(\xi)^\beta} (u_n-u)\psi~\mathrm{d}\xi=0. 
\end{equation}
 First notice that \eqref{eq3.19} holds by using the same idea as in Proposition \ref{prop3.5}. By employing the H\"older's inequality, to complete the proof of \eqref{eq3.29}, it is sufficient to show that 
\begin{align}\label{eq3.30}
    \|f(\cdot,u_n)(u_n-u)\psi\|_{{\frac{2Q}{2Q-2\beta-\lambda},~ B_{2R}}}\to 0~~\text{as}~~n\to\infty.
\end{align}
In view of \eqref{eq3.17}, there exists $n_0\in\mathbb{N}$ and $\delta>0$ such that $\|u_n\|^{Q^\prime}_{HW^{1,Q}}<\delta<\alpha_Q$ for all $n\geq n_0$. Suppose that $\zeta$, $\zeta^\prime>1$ satisfying $\frac{1}{\zeta}+\frac{1}{\zeta^\prime}=1$. Choose $\alpha>\alpha_0$ very close to $\alpha_0$ and $\zeta^\prime>1$ very close to 1 such that we still have $\displaystyle\frac{2\alpha Q \zeta^\prime }{2Q-2\beta-\lambda} \|u_n\|^{Q^\prime}_{HW^{1, Q}}<\delta<\alpha_Q$ for all $n\geq n_0$. Consequently, arguing similarly as in Proposition \ref{prop3.5}, we have
\begin{align*}
  \int_{B_{2R}}|f(\xi,u_n)(u_n-u)\psi|^{\frac{2Q}{2Q-2\beta-\lambda}}~\mathrm{d}\xi&\leq  C \Bigg[\bigg(\int_{B_{2R}}|u_n-u|^{\frac{2Q^2}{2Q-2\beta-\lambda}}~\mathrm{d}\xi\bigg)^{\frac{1}{Q}}  +\bigg(\int_{B_{2R}}|u_n-u|^{\frac{2Q\zeta}{2Q-2\beta-\lambda}}~\mathrm{d}\xi\bigg)^{\frac{1}{\zeta}}\Bigg] \notag\\
    & \qquad \to 0~~\text{as}~~n\to\infty,
\end{align*}
we thank to \eqref{eq3.24}, where for $\widetilde{u}_n=u_n/{\|u_n\|_{HW^{1,Q}}}$, the constant $C$ is defined by
\begin{align*}
    C&=\max\Bigg\{ \sup_{n\in\mathbb{N}}\|u_n\|^{\frac{2Q(Q-1)}{2Q-2\beta-\lambda}}_{\mathbf{X}},~\bigg(\sup_{n\geq n_0}\int_{\mathbb{H}^N}\Phi\bigg(\frac{2\alpha Q \zeta^\prime \|u_n\|^{Q^\prime}_{HW^{1, Q}}}{2Q-2\beta-\lambda}|\widetilde{u}_n|^{Q^\prime} \bigg)~\mathrm{d}\xi\bigg)^{\frac{1}{\zeta^\prime}}~\Bigg\}<+\infty,
\end{align*}
due to the uniform boundedness of $\{u_n\}_n$ in $\mathbf{X}$ and Theorem \ref{thm2.7}. This shows that \eqref{eq3.30} holds and hence we obtain the claim \eqref{eq3.29}. Similarly, we can prove that
\begin{equation}\label{eq3.31}
   \lim_{n\to\infty} \displaystyle\int_{\mathbb{H}^N}\Bigg(\displaystyle\int_{\mathbb{H}^N}\frac{F(\eta,u)}{r(\eta)^\beta {d_K(\xi,\eta)}^\lambda}~\mathrm{d}\eta\Bigg)\frac{f(\xi,u)}{r(\xi)^\beta} (u_n-u)\psi~\mathrm{d}\xi=0. 
\end{equation}
Passing $n\to\infty$ in \eqref{eq3.26} and using the convergence results \eqref{eq3.27}--\eqref{eq3.29} and \eqref{eq3.31}, we get $u_n\to u$ in $HW^{1,Q}(B_{R})$ as $n\to\infty$. Consequently, we get $D_H u_n\to D_H u$ in $L^{Q}(B_{R},\mathbb{R}^{2N})$ as $n\to\infty$ for all $R>0$. Hence, up to a subsequence, not relabelled, we obtain $D_H u_n\to D_H u$  a.e. in $\mathbb{H}^N$ as $n\to\infty$. This together with \eqref{eq3.24} imply that $|D_H u_n|^{t-1}D_H u_n\to |D_H u|^{t-1}D_H u$ a.e. in $\mathbb{H}^N$ and $|u_n|^{t-1} u_n\to | u|^{t-1} u$ a.e. in $\mathbb{H}^N$ as $n\to\infty$ for $t\in\{p,Q\}$. Now using the facts that $\{|D_H u_n|^{t-1}D_H u_n\}_n$ and $\{|u_n|^{t-1} u_n\}_n$ are bounded in $L^{\frac{t}{t-1}}(\mathbb{H}^N,\mathbb{R}^{2N})~\text{and}~L^{\frac{t}{t-1}}(\mathbb{H}^N)$ respectively for $t\in\{p, Q\}$, we deduce that \eqref{eq3.333} holds. This ends the proof of Proposition \ref{prop3.6}. 
\end{proof}
\begin{proof}[\textbf{Proof of Theorem \ref{thm1.1}}.]
Let $\gamma>0$ be fixed and $\mu\in (0,\widehat{\mu})$, with $\widehat{\mu}=\text{min}\{\mu^\ast,\mu^0\}$, where $\mu^\ast$ as in Lemma \ref{lem3.2}, while we define $\mu^0$ by
\begin{equation}\label{eq3.33}
 \mu^0=\frac{s(\sigma-Q)}{Q\|g\|_\vartheta (\sigma-s)} \alpha^{\frac{Q-s}{Q^\prime}}_Q >0. 
\end{equation}
  Due to Lemma \ref{lem3.4}, there exists a sequence $\{u_n\}_n$ of nonnegative functions in $\overline{B}_\rho$ such that \eqref{eq3.11} holds. Now replacing $u$ with $u_{\mu,\gamma}$ in \eqref{eq3.99}, one can observe that \eqref{eq3.99} still holds. Moreover, by using Lemma \ref{lem2.4}, Lemma \ref{lem3.4} and $(f3)$, we obtain as $n\to\infty$
\begin{align*}
0&>m_{\mu,\gamma}= J(u_n)-\frac{1}{\sigma}\langle{J^\prime(u_n),u_n}\rangle+o_n(1)\\
&\geq \bigg(\frac{1}{p}-\frac{1}{\sigma}\bigg)\|u_n\|^{p}_{HW^{1,p}}+\bigg(\frac{1}{Q}-\frac{1}{\sigma}\bigg)\|u_n\|^{Q}_{HW^{1,Q}}-\mu\bigg(\frac{1}{s}-\frac{1}{\sigma}\bigg)\|g\|_\vartheta\|u_n\|^s_{HW^{1,Q}}+o_n(1)  \\
&\geq\bigg(\frac{1}{Q}-\frac{1}{\sigma}\bigg)\|u_n\|^{Q}_{HW^{1,Q}}-\mu\bigg(\frac{1}{s}-\frac{1}{\sigma}\bigg)\|g\|_\vartheta\|u_n\|^s_{HW^{1,Q}}+o_n(1). \end{align*}
This together with \eqref{eq3.33} yields
$$\limsup_{n\to\infty}\|u_n\|^{Q^\prime}_{HW^{1,Q}}\leq\Bigg(\frac{\mu(\sigma-s)Q\|g\|_\vartheta}{s(s-Q)}\Bigg)^{\frac{Q^\prime}{Q-s}}< \alpha_Q.   $$
In light of the above fact, we can argue as in the proof of Proposition \ref{prop3.6} to deduce that $D_H u_n\to D_H u_{\mu,\gamma}$ a.e. in $\mathbb{H}^N$ as $n\to\infty$. Further, notice that \eqref{eq3.333} holds whenever $u$ is replaced by $u_{\mu,\gamma}$. Choose $v\in\mathcal{C}^\infty_c(\mathbb{H}^N)$, then exploiting the density of $\mathcal{C}^\infty_c(\mathbb{H}^N)$ in $HW^{1,\wp}(\mathbb{H}^N)$ for any $\wp\in(1,\infty)$ , we obtain for $t\in\{p,Q\}$ that 
\begin{equation}\label{eq3.35}
  \int_{\mathbb{H}^N}|D_H u_n|_H^{t-2}(D_H u_n,D_H v)_H~\mathrm{d}\xi\to \int_{\mathbb{H}^N}|D_H u_{\mu,\gamma}|_H^{t-2}(D_H u_{\mu,\gamma},D_H v)_H~\mathrm{d}\xi~~\text{as}~n\to\infty  
\end{equation}
and 
\begin{equation}\label{eq3.36}
  \int_{\mathbb{H}^N}u_n^{t-1} v~\mathrm{d}\xi\to \int_{\mathbb{H}^N}u_{\mu,\gamma}^{t-1}v~\mathrm{d}\xi~~\text{as}~n\to\infty.  
\end{equation}
Because of Lemma \ref{lem2.4}, we deduce that $\{u_n\}_n$ is bounded in $L^s_g(\mathbb{H}^N)$. It follows at once that $\{g(\cdot)^{\frac{s-1}{s}}u_n^{s-1}\}_n$ is bounded in $L^{\frac{s-1}{s}}(\mathbb{H}^N)$. Consequently, using the facts that $g(\cdot)^{\frac{s-1}{s}}u_n^{s-1}\to g(\cdot)^{\frac{s-1}{s}}u_{\mu,\gamma}^{s-1}$ a.e. in $\mathbb{H}^N$, we have 
$$ g(\cdot)^{\frac{s-1}{s}}u_n^{s-1}\rightharpoonup g(\cdot)^{\frac{s-1}{s}}u_{\mu,\gamma}^{s-1}~\text{in}~L^{\frac{s-1}{s}}(\mathbb{H}^N)~\text{as}~n\to\infty.$$
Once more by Lemma \ref{lem2.4} and the density of $\mathcal{C}^\infty_c(\mathbb{H}^N)$ in $HW^{1,Q}(\mathbb{H}^N)$, we can see that $g(\cdot)^{\frac{1}{s}}v\in L^{s}(\mathbb{H}^N)$. Therefore, we get  
\begin{equation}\label{eq3.37}
  \int_{\mathbb{H}^N}g(\xi)u_n^{s-1} v~\mathrm{d}\xi\to \int_{\mathbb{H}^N} g(\xi)u_{\mu,\gamma}^{s-1}v~\mathrm{d}\xi~~\text{as}~n\to\infty.  
\end{equation}
By Proposition \ref{prop3.5}, we also have
\begin{align}\label{eq3.38}
\displaystyle\int_{\mathbb{H}^N}\Bigg(\displaystyle\int_{\mathbb{H}^N}\frac{F(\eta,u_n)}{r(\eta)^\beta {d_K(\xi,\eta)}^\lambda}~\mathrm{d}\eta\Bigg)\frac{f(\xi,u_n)}{r(\xi)^\beta} v~\mathrm{d}\xi\to\displaystyle\int_{\mathbb{H}^N}\Bigg(\displaystyle\int_{\mathbb{H}^N}\frac{F(\eta,u_{\mu,\gamma})}{r(\eta)^\beta {d_K(\xi,\eta)}^\lambda}~\mathrm{d}\eta\Bigg)\frac{f(\xi,u_{\mu,\gamma})}{r(\xi)^\beta} v~\mathrm{d}\xi~\text{as}~n\to\infty.
\end{align}
Finally, due the convergence results of \eqref{eq3.35}--\eqref{eq3.38} and using the facts that $\langle{J^\prime(u_n),v}\rangle=o_n(1)~\text{as}~n\to\infty$, we deduce that $\langle{J^\prime(u_{\mu,\gamma}),v}\rangle=0$, for all $v\in\mathcal{C}^\infty_c(\mathbb{H}^N)$.

Suppose $\varrho\in \mathcal{C}^\infty_c(\mathbb{H}^N)$ such that $0\leq \varrho\leq 1,~\varrho\equiv 1~\text{in}~B_1 $ and $\text{supp}(\varrho)\subset B_2$. Define the sequence of cut-off functions $$\varrho_n(\xi)=\varrho\big(\delta_{\frac{1}{n}}(\xi)\big), ~\xi\in \mathbb{H}^N, $$
where $\delta_{\frac{1}{n}}$ is the dilation of parameter $\frac{1}{n}$ as in \eqref{eq2.2}. Now choosing $\varphi\in\mathbf{X}$, then the sequence $\{\varphi_n\}_n$, which is defined by $\varphi_n=\varrho_n\big(\rho_n\ast\varphi\big)$, where $\{\rho_n\}_n$ is the sequence of mollifiers and $\{\varrho_n\}_n$ is the sequence of cut-off functions with the property that $\varphi_n\to\varphi$ in $\mathbf{X}$ as $n\to\infty$. This shows that up to the subsequence $\varphi_n\to\varphi$, $D_H \varphi_n\to D_H \varphi$ a.e. in $\mathbb{H}^N$ as $n\to\infty$, and there exists functions $g_1\in L^p(\mathbb{H}^N)$ and $g_2\in L^Q(\mathbb{H}^N)$ such that $|\varphi_n|\leq g_i$, $|D_H \varphi_n|\leq g_i$ a.e. in $\mathbb{H}^N$ for all $n$ and $i=1,2$. Therefore, it follows that $\langle{J^\prime(u_{\mu,\gamma}), \varphi_n}\rangle=0$ for all $n\in\mathbb{N}$ and thus passing to the limit as $n\to\infty$ under the sign of integrals along with the Lebesgue's dominated convergence theorem or Vitali's convergence theorem, we obtain that $\langle{J^\prime(u_{\mu,\gamma}), \varphi}\rangle=0$ for all $\varphi\in\mathbf{X}$. Consequently, we obtain that $u_{\mu,\gamma}$ is a solution of \eqref{main problem}.

Moreover, since $\{u_n\}_n$ is bounded in $HW^{1,t}(\mathbb{H}^N)$ for $t\in\{p, Q\}$, therefore up to a subsequence, not relabelled, there exists $l_p,l_Q\geq 0$ such that we have $\|u_n\|_{HW^{1,t}}\to l_t$ in $\mathbb{R}$ as $n\to\infty$. By using Lemma \ref{lem2.4} and \eqref{eq3.99}, we infer that 
\begin{equation}\label{eq3.39}
  \int_{\mathbb{H}^N}g(\xi)u_n^{s} ~\mathrm{d}\xi\to \int_{\mathbb{H}^N} g(\xi)u_{\mu,\gamma}^{s}~\mathrm{d}\xi~~\text{as}~n\to\infty.    
\end{equation}
Define $$I(u)= \displaystyle\int_{\mathbb{H}^N}\Bigg(\displaystyle\int_{\mathbb{H}^N}\frac{F(\eta,u)}{r(\eta)^\beta {d_K(\xi,\eta)}^\lambda}~\mathrm{d}\eta\Bigg)\bigg(\frac{2f(\xi,u)u-\sigma F(\xi,u)}{r(\xi)^\beta} \bigg)~\mathrm{d}\xi.$$
Then, it follows from \eqref{eq3.11}, $(f2)$ and the Fatou's lemma that
\begin{equation}\label{eq3.40}
 \liminf_{n\to\infty} I(u_n)\geq I(u_{\mu,\gamma}).
\end{equation}
In view of \eqref{eq3.11}, \eqref{eq3.39} and \eqref{eq3.40}, we obtain
\begin{align*}    m_{\mu,\gamma}&=\lim_{n\to\infty}\bigg( J(u_n)-\frac{1}{\sigma}\langle{J^\prime(u_n),u_n}\rangle\bigg) \notag\\
&= \bigg(\frac{1}{p}-\frac{1}{\sigma}\bigg)l_p^{p}+\bigg(\frac{1}{Q}-\frac{1}{\sigma}\bigg)l_Q^{Q}-\mu\bigg(\frac{1}{s}-\frac{1}{\sigma}\bigg)\|u_{\mu,\gamma}\|^s_{s,g}+\frac{\gamma}{2\sigma} \liminf_{n\to\infty} I(u_n) \notag\\
&\geq \bigg(\frac{1}{p}-\frac{1}{\sigma}\bigg)\|u_{\mu,\gamma}\|^{p}_{HW^{1,p}}+\bigg(\frac{1}{Q}-\frac{1}{\sigma}\bigg)\|u_{\mu,\gamma}\|^{Q}_{HW^{1,Q}}-\mu\bigg(\frac{1}{s}-\frac{1}{\sigma}\bigg)\|u_{\mu,\gamma}\|^s_{s,g}+\frac{\gamma}{2\sigma}I(u_{\mu,\gamma}) \notag\\
&=J(u_{\mu,\gamma})-\frac{1}{\sigma}\langle{J^\prime(u_{\mu,\gamma}),u_{\mu,\gamma}}\rangle=J(u_{\mu,\gamma})\geq m_{\mu,\gamma}. 
\end{align*}
This shows that the solution $u_{\mu,\gamma}$ is a minimizer of $J$ in $\overline{B}_\rho$. Consequently, we also have $J(u_{\mu,\gamma})=m_{\mu,\gamma}<0<\jmath\leq J(u)$ for $\|u\|_{\mathbf{X}}=\rho$, thanks to Lemma \ref{lem3.2}. Therefore, we conclude that $u_{\mu,\gamma}$ is in $B_\rho$, which is a nontrivial solution of \eqref{main problem} for all $\mu\in (0,\widehat{\mu})$. 

Notice that $u_{\mu,\gamma}\in B_\rho$ with $\rho>0$ independent of $\mu$ and hence $\{u_{\mu,\gamma}\}_{\mu\in (0,\widehat{\mu})}$ is uniformly bounded in $\mathbf{X}$. Further, from \eqref{eq3.11} and $(f2)$, we get as $n\to\infty$
\begin{align*}
 0>m_{\mu,\gamma}&\geq \bigg(\frac{1}{p}-\frac{1}{\sigma}\bigg)\|u_{\mu,\gamma}\|^{p}_{HW^{1,p}}+\bigg(\frac{1}{Q}-\frac{1}{\sigma}\bigg)\|u_{\mu,\gamma}\|^{Q}_{HW^{1,Q}}-\mu\bigg(\frac{1}{s}-\frac{1}{\sigma}\bigg)\|g\|_\vartheta\|u_{\mu,\gamma}\|^s_{HW^{1,Q}} \notag\\
 & \geq \bigg(\frac{1}{p}-\frac{1}{\sigma}\bigg)\|u_{\mu,\gamma}\|^{p}_{HW^{1,p}}+\bigg(\frac{1}{Q}-\frac{1}{\sigma}\bigg)\|u_{\mu,\gamma}\|^{Q}_{HW^{1,Q}}-\mu C_g,
\end{align*}
where $$C_g=\bigg(\frac{1}{s}-\frac{1}{\sigma}\bigg)\|g\|_\vartheta \sup_{\mu\in (0,\widehat{\mu})}\|u_{\mu,\gamma}\|^s_{\mathbf{X}}<\infty. $$
This yields
$$0\geq \limsup_{\mu\to 0^+} m_{\mu,\gamma}\geq \limsup_{\mu\to 0^+}\Bigg[\bigg(\frac{1}{p}-\frac{1}{\sigma}\bigg)\|u_{\mu,\gamma}\|^{p}_{HW^{1,p}}+\bigg(\frac{1}{Q}-\frac{1}{\sigma}\bigg)\|u_{\mu,\gamma}\|^{Q}_{HW^{1,Q}}\Bigg]\geq 0. $$
It follows that $$\lim_{\mu\to 0^+}\|u_{\mu,\gamma}\|_{HW^{1,p}}=\lim_{\mu\to 0^+}\|u_{\mu,\gamma}\|_{HW^{1,Q}}=0,~\text{i.e.,}~\lim_{\mu\to 0^+}\|u_{\mu,\gamma}\|_{\mathbf{X}}=0.$$
This completes the proof of Theorem \ref{thm1.1}. 
\end{proof}
\section{Existence of the second solution}\label{sec4}
For the sake of simplicity, we believe in this section that the structural assumptions required for Theorem \ref{thm1.2} are valid.
It is evident from Lemma \ref{lem3.2} and Lemma \ref{lem3.3} that the energy functional $J$ fulfils the geometrical structures of the mountain pass theorem.
In order to implement the mountain pass theorem, it is necessary to verify the validity of the Palais-Smale compactness condition at an appropriate level c. We say $\{u_n\}_n\subset\mathbf{X}$ is a $(PS)_c$ sequence for $J$ at any suitable level $c$ if 
\begin{equation}\label{eq4.1}
  J(u_n)\to c~\text{and}~\sup_{\|\varphi\|_{\mathbf{X}}=1}|\langle{J^\prime(u_n),\varphi}\rangle|\to 0~\text{as}~n\to\infty.  
\end{equation}
Note that $J$ meets the $(PS)_c$ condition at any suitable level $c$ if this sequence has a convergent subsequence in $\mathbf{X}$.
\begin{lemma}\label{lem4.1}
 Suppose $\mu,\gamma>0$ be fixed. Then, the $(PS)_c$ sequence $\{u_n\}_n\subset\mathbf{X}$ for $J$ at any level $c>0$ is bounded in $\mathbf{X}$ and satisfying
 \begin{equation}\label{eq4.444}
   c+\Theta\mu^{\vartheta}+o_n(1)\geq \boldsymbol{\kappa} \big(\|u_n\|^{p}_{HW^{1,p}}+\|u_n\|^{Q}_{HW^{1,Q}}\big)~\text{as}~n\to\infty,   
 \end{equation}
 where  $\Theta$ and $\boldsymbol{\kappa} $ are defined by
 $$\Theta=\Big(\frac{\sigma-s}{Q}\Big)^{\frac{Q}{Q-s}} \Big(\frac{2Q}{\sigma-Q}\Big)^{\frac{s}{Q-s}}\frac{(Q-s)\|g\|^\vartheta_\vartheta}{\sigma s},~\text{and}~ \boldsymbol{\kappa}=\frac{\sigma-Q}{2Q\sigma}.$$
\end{lemma}
\begin{proof}
Let $\{u_n\}_n\subset\mathbf{X}$ be a $(PS)_c$ sequence for $J$ at level $c>0$ and hence \eqref{eq4.1} is satisfied. Therefore, by using $(f2)$, the H\"older's inequality and Young's inequality with $\epsilon$, we have as $n\to\infty$,
\begin{align*}c+o_n(1)+o_n(1)\|u_n\|_{\mathbf{X}}&=J(u_n)-\frac{1}{\sigma}\langle{J^\prime(u_n),u_n}\rangle \\
&\geq \bigg(\frac{1}{p}-\frac{1}{\sigma}\bigg)\|u_n\|^{p}_{HW^{1,p}}+\bigg(\frac{1}{Q}-\frac{1}{\sigma}\bigg)\|u_n\|^{Q}_{HW^{1,Q}}-\mu\bigg(\frac{1}{s}-\frac{1}{\sigma}\bigg)\|g\|_\vartheta\|u_n\|^{s}_{HW^{1,Q}} \\
&\geq \bigg(\frac{1}{p}-\frac{1}{\sigma}\bigg)\|u_n\|^{p}_{HW^{1,p}}+\bigg[\bigg(\frac{1}{Q}-\frac{1}{\sigma}\bigg)-\epsilon \bigg(\frac{1}{s}-\frac{1}{\sigma}\bigg)\bigg]\|u_n\|^{Q}_{HW^{1,Q}}-C_\epsilon\bigg(\frac{1}{s}-\frac{1}{\sigma}\bigg)\mu^\vartheta\|g\|^\vartheta_{\vartheta}.\end{align*}
Choose $\epsilon=\big(\frac{1}{Q}-\frac{1}{\sigma}\big)\big/2\big(\frac{1}{s}-\frac{1}{\sigma}\big)$, $\Theta=C_\epsilon\bigg(\frac{1}{s}-\frac{1}{\sigma}\bigg)\|g\|^\vartheta_{\vartheta}$, and $\boldsymbol{\kappa}=\frac{\sigma-Q}{2Q\sigma}$, then we get from the above inequality that
\begin{equation}\label{eq4.2}
 c+\Theta \mu^\vartheta+ o_n(1)+o_n(1)\|u_n\|_{\mathbf{X}}\geq \boldsymbol{\kappa}\big(\|u_n\|^{p}_{HW^{1,p}}+\|u_n\|^{Q}_{HW^{1,Q}}\big)~\text{as}~n\to\infty. 
 \end{equation}
If possible, let $\{u_n\}_n$ be not bounded in $\mathbf{X}$, then we  have following possibilities:\\
\textbf{Case-1:}
Let $\|u_n\|_{HW^{1,p}}\to\infty$ and $\|u_n\|_{HW^{1,Q}}\to\infty$ as $n\to\infty$. Observe that $p<Q$, therefore we have $\|u_n\|^{Q}_{HW^{1,Q}}\geq \|u_n\|^{p}_{HW^{1,Q}}>1$ for $n$ large enough. It follows from \eqref{eq4.2} that
$$ c+\Theta \mu^\vartheta+ o_n(1)+o_n(1)\|u_n\|_{\mathbf{X}}\geq 2^{1-p}\boldsymbol{\kappa}\|u_n\|^p_\mathbf{X}~\text{as}~n\to\infty.$$
Dividing $\|u_n\|^{p}_{\mathbf{X}}$ on both the sides and letting $n\to\infty$, we get $0\geq 2^{1-p}\boldsymbol{\kappa}>0 $, which is a contradiction.\\
\textbf{Case-2:}
Let $\|u_n\|_{HW^{1,p}}\to\infty$ as $n\to\infty$ and $\|u_n\|_{W^{1,Q}}$ be bounded. From \eqref{eq4.2}, we infer that
$$ c+\Theta \mu^\vartheta+ o_n(1)+o_n(1)\|u_n\|_{\mathbf{X}}\geq \boldsymbol{\kappa}\|u_n\|^{p}_{HW^{1,p}}~\text{as}~n\to\infty.$$
Dividing $\|u_n\|^{p}_{HW^{1,p}}$ on both the sides and letting $n\to\infty$, we get $0\geq \boldsymbol{\kappa}>0 $, which is again a contradiction.\\
\textbf{Case-3:}
Let $\|u_n\|_{HW^{1,Q}}\to\infty$ as $n\to\infty$ and $\|u_n\|_{HW^{1,p}}$ be bounded. Similarly, we get a contradiction as in \textit{Case-2}.

Thus, we conclude from the above three situations that $\{u_n\}_n\subset\mathbf{X}$ must be bounded. Therefore, taking into account \eqref{eq4.2}, the proof of the lemma is completed.
\end{proof}
The following result is a direct consequence of Lemma \ref{lem4.1}.
\begin{corollary}\label{cor4.2}
 Let $\mu\in(0,\mu^\ast]$, where $\mu^\ast$ as in Lemma \ref{lem3.2} and $\gamma>0$ be fixed. Now, fix $c<c_0-\Theta\mu^{\vartheta}$, where $c_0>0$ satisfying
\begin{equation}\label{eq4.3}
  c_0\leq \boldsymbol{\kappa}\alpha^{Q-1}_Q,   
\end{equation}
where $\boldsymbol{\kappa}$, and $\Theta$ are defined as in Lemma \ref{lem4.1}. Then, for any $(PS)_c$ sequence $\{u_n\}_n\subset\mathbf{X}$ of $J$ at level $c>0$, there holds $\displaystyle\limsup_{n\to\infty}\|u_n\|^{Q^\prime}_{HW^{1,Q}}<\alpha_Q$. In addition, if $\alpha>\alpha_0$, where $\alpha_0$ as in $(f3)$, then there exists $n_0\in\mathbb{N}$ sufficiently large enough such that
\begin{align*}
 \sup_{n\geq n_0}\Bigg|\displaystyle\int_{\mathbb{H}^N}\Bigg(\displaystyle\int_{\mathbb{H}^N}\frac{F(\eta,u_n)}{r(\eta)^\beta {d_K(\xi,\eta)}^\lambda}~\mathrm{d}\eta\Bigg)\frac{f(\xi,u_n)}{r(\xi)^\beta} u_n~\mathrm{d}\xi\Bigg|<+\infty.    
\end{align*}   
\end{corollary}
Following in Br\'ezis-Lieb's footsteps, the next lemma is dedicated to doubly weighted Choquard nonlinearity.
\begin{proposition}\label{prop4.3}    
 Let $\{u_n\}_n\subset \mathbf{X} $ be a sequence and $u$ be in $\mathbf{X}$ such that $u_n\rightharpoonup u$ in $\mathbf{X}$, $u_n\to u$ a.e. in $\mathbb{H}^N$ and $D_H u_n\to D_H u$ a.e. in $\mathbb{H}^N$ as $n\to\infty$. If
 \begin{equation}\label{eq4.4}
 \limsup_{n\to\infty}\|u_n\|^{Q^\prime}_{HW^{1,Q}}<\frac{\alpha_Q}{2^{Q^\prime}}
\end{equation}
is satisfied, then for $\boldsymbol{t}=\frac{2Q}{2Q-2\beta-\lambda}$, we have as $n\to\infty$
\begin{itemize}
   
    \item [(i)] $$\int_{\mathbb{H}^N}\big|F(\xi,u_n)-F(\xi,u_n-u)-F(\xi,u)\big|^{\boldsymbol{t}}~\mathrm{d}\xi=o_n(1);$$
     \item [(ii)] $$\int_{\mathbb{H}^N}\big|f(\xi,u_n)u_n-f(\xi,u_n-u)(u_n-u)-f(\xi,u)u\big|^{\boldsymbol{t}}~\mathrm{d}\xi=o_n(1).$$
\end{itemize}
Furthermore, if we define
    $$ \mathcal{G}(u)=\displaystyle\int_{\mathbb{H}^N}\Bigg(\displaystyle\int_{\mathbb{H}^N}\frac{F(\eta,u)}{r(\eta)^\beta {d_K(\xi,\eta)}^\lambda}~\mathrm{d}\eta\Bigg)\frac{f(\xi,u)}{r(\xi)^\beta} u~\mathrm{d}\xi~~\text{and }~~\mathcal{T}(u)=\displaystyle\int_{\mathbb{H}^N}\Bigg(\displaystyle\int_{\mathbb{H}^N}\frac{F(\eta,u)}{r(\eta)^\beta {d_K(\xi,\eta)}^\lambda}~\mathrm{d}\eta\Bigg)\frac{F(\xi,u)}{r(\xi)^\beta}~\mathrm{d}\xi, $$
    then the following results hold as $n\to\infty$
   \begin{itemize}
       \item [(a)] $$\mathcal{G}(u_n)-\mathcal{G}(u_n-u)-\mathcal{G}(u)=o_n(1);$$
       \item [(b)] $$\mathcal{T}(u_n)-\mathcal{T}(u_n-u)-\mathcal{T}(u)=o_n(1) .$$
   \end{itemize}
\end{proposition}
\begin{proof}
Set $v_n=u_n-u$ and $H_n(\cdot)=|F(\cdot,v_n+u)-F(\cdot,v_n)-F(\cdot,u)|^{\boldsymbol{t}}$. By using $(f1)$ and Corollary \ref{cor2.5}, we obtain
\begin{align*}
    |F(\xi,v_n+u)-F(\xi,v_n)|&=\bigg|\int_{0}^{1}\frac{\mathrm{d}}{\mathrm{d}s}\big(F(\xi,v_n+su)\big)~\mathrm{d}s\bigg|\leq \int_{0}^{1} |f(\xi,v_n+su)||u|~\mathrm{d}s\\
    &\leq \int_{0}^{1}|u|\Big( \varepsilon  |v_n+su|^{Q-1}+C_\varepsilon  \Phi(\alpha_0 |v_n+su|^{Q^\prime})\Big)~\mathrm{d}s\\
    & \leq  2^{Q-2} \varepsilon \int_{0}^{1} \Big(|v_n|^{Q-1}+|su|^{Q-1}\Big)|u|~\mathrm{d}s+C_\varepsilon \int_{0}^{1}|u|\Phi\Big(\alpha_0 \big(|v_n|+s|u|\big)^{Q^\prime}\Big)~\mathrm{d}s\\
    & \leq  2^{Q-2} \varepsilon \Big[|v_n|^{Q-1}|u|+|u|^Q\Big]+C_\varepsilon |u|\Phi\Big(\alpha_0 \big(|v_n|+|u|\big)^{Q^\prime}\Big).
\end{align*}
It follows that
 $$|F(\xi,v_n+u)-F(\xi,v_n)-F(\xi,u)|\leq  2^{Q-2} \varepsilon |v_n|^{Q-1}|u|+\big(2^{Q-2}+1 \big)\varepsilon |u|^Q+C_\varepsilon |u|\big[\Phi\big(\alpha_0 \big(|v_n|+|u|\big)^{Q^\prime}\big)+ \Phi\big(\alpha_0 |u|^{Q^\prime}\big)\big].$$
This together with Corollary \ref{cor2.5} implies that there exists a constant $C>0$ such that    
\begin{align}\label{eq4.5}
 H_n(\xi)\leq C \big[|v_n|^{(Q-1)\boldsymbol{t}}|u|^{\boldsymbol{t}}+|u|^{Q\boldsymbol{t}}+|u|^{\boldsymbol{t}}\big(\Phi\big(\alpha_0 \boldsymbol{t}\big(|v_n|+|u|\big)^{Q^\prime}\big)+ \Phi\big(\alpha_0 \boldsymbol{t}|u|^{Q^\prime}\big)\big)\big].   
\end{align} 
Notice that $u_n\rightharpoonup u$ in $HW^{1,Q}(\mathbb{H}^N)$ as $n\to\infty$ and thus $\{u_n\}_n$ is bounded in $HW^{1,Q}(\mathbb{H}^N)$. It follows from Br\'ezis-Lieb lemma that
$$\|v_n\|^Q_{HW^{1,Q}}=\|u_n\|^Q_{HW^{1,Q}}-\|u\|^Q_{HW^{1,Q}}+o_n(1)\leq \|u_n\|^Q_{HW^{1,Q}}+o_n(1)~\text{as}~n\to\infty. $$
Consequently, by using \eqref{eq4.4}, we have
\begin{equation}\label{eq4.6}
  \limsup_{n\to\infty}\|v_n\|^{Q^\prime}_{HW^{1,Q}}<\frac{\alpha_Q}{2^{Q^\prime}}.   
\end{equation}
Due to weakly lower semicontinuity of the norm, we obtain
\begin{equation}\label{eq4.7}
 \|u\|^{Q^\prime}_{HW^{1,Q}}\leq \liminf_{n\to\infty}\|u_n\|^{Q^\prime}_{HW^{1,Q}}\leq \limsup_{n\to\infty}\|u_n\|^{Q^\prime}_{HW^{1,Q}}<\frac{\alpha_Q}{2^{Q^\prime}}. 
\end{equation}
Combining \eqref{eq4.6} and \eqref{eq4.7} together, we can easily deduce that
\begin{equation}\label{eq4.8}
  \limsup_{n\to\infty}\big\||v_n|+|u|\big\|^{Q^\prime}_{HW^{1,Q}}<\alpha_Q.   
\end{equation}
Taking into account \eqref{eq4.8} there exist $m>0$ and $n_0\in\mathbb{N}$ such that $\big\||v_n|+|u|\big\|^{Q^\prime}_{HW^{1,Q}}<m<\alpha_Q$ for all $n\ge n_0$. Suppose that $\zeta\geq Q$ with $\zeta^\prime=\frac{\zeta}{1-\zeta}$ satisfying $\frac{1}{\zeta}+\frac{1}{\zeta^\prime}=1$. Choose $\alpha>\alpha_0$ very close to $\alpha_0$ and $\zeta^\prime>1$ very close to 1 such that we still have  $\zeta^\prime\alpha \boldsymbol{t}\big\||v_n|+|u|\big\|^{Q^\prime}_{HW^{1,Q}}<m<\alpha_Q$ for all $n\ge n_0$. Denote $\widetilde{v}_n=|v_n|+|u|\bigg/\big\||v_n|+|u|\big\|_{HW^{1,Q}}$, then by using the H\"older's inequality, Corollary \ref{cor2.3} and Corollary \ref{cor2.5}, it follows from  \eqref{eq4.5} that for $n\geq n_0$, we have
\begin{align*}\int_{\mathbb{H}^N}H_n(\xi)~\mathrm{d}\xi &\leq C\Bigg[ \|v_n\|^{(Q-1)\boldsymbol{t}}_{Q\boldsymbol{t}}\|u\|_{Q\boldsymbol{t}}+\|u\|^{Q\boldsymbol{t}}_{Q\boldsymbol{t}}+\|u\|^{\boldsymbol{t}}_{\boldsymbol{t}\zeta}\Bigg(\int_{\mathbb{H}^N}\Phi\big(\zeta^\prime\alpha\boldsymbol{t}\big\||v_n|+|u|\big\|^{Q^\prime}_{HW^{1,Q}}|\widetilde{v}_n|^{N^\prime}\big)~\mathrm{d}\xi\Bigg)^{\frac{1}{\zeta^\prime}}\\
&\qquad~~+\|u\|^{\boldsymbol{t}}_{\boldsymbol{t}\zeta}\Bigg(\int_{\mathbb{H}^N}\Phi\big(\zeta^\prime\alpha \boldsymbol{t}|u|^{Q^\prime}\big)~\mathrm{d}\xi\Bigg)^{\frac{1}{\zeta^\prime}}\Bigg].
\end{align*}
This yields that for all $n\geq n_0$, we have
\begin{equation}\label{eq4.9}
 \int_{\mathbb{H}^N}H_n(\xi)~\mathrm{d}\xi\leq C_1\Bigg[\bigg(\int_{\mathbb{H}^N}|u|^{Q\boldsymbol{t}}~\mathrm{d}\xi\bigg)^{\frac{1}{Q\boldsymbol{t}}}+\int_{\mathbb{H}^N}|u|^{Q\boldsymbol{t}}~\mathrm{d}\xi+\bigg(\int_{\mathbb{H}^N}|u|^{\boldsymbol{t}\zeta}~\mathrm{d}\xi\bigg)^{\frac{1}{\zeta}}\Bigg] <\infty,  
\end{equation}
where \begin{align*}C_1&= C ~\text{max}\Bigg\{\sup_{n\in\mathbb{N}}\|v_n\|^{(Q-1)\boldsymbol{t}}_{Q\boldsymbol{t}},1, \Bigg(\sup_{n\geq n_0}\int_{\mathbb{H}^N}\Phi\big(\zeta^\prime\alpha\boldsymbol{t}\big\||v_n|+|u|\big\|^{Q^\prime}_{HW^{1,Q}}|\widetilde{v}_n|^{N^\prime}\big)~\mathrm{d}\xi\Bigg)^{\frac{1}{\zeta^\prime}}+\Bigg(\int_{\mathbb{H}^N}\Phi\big(\zeta^\prime\alpha \boldsymbol{t}|u|^{Q^\prime}\big)~\mathrm{d}\xi\Bigg)^{\frac{1}{\zeta^\prime}}\Bigg\},\end{align*}
which is finite, thanks to Theorem \ref{thm2.7} and the fact that $\{v_n\}_n$ is uniformly bounded in $\mathbf{X}$. It follows that $\{H_n(\cdot)\}_{n\geq n_0}$ is bounded in $L^1(\mathbb{H}^N)$. Moreover, because of \eqref{eq4.9}, it is not difficult to verify that $\{H_n(\cdot)\}_{n\geq n_0}$ is uniformly absolutely integrable and tight over $\mathbb{H}^N$. Further, since $u_n\to u$ a.e. in $\mathbb{H}^N$ as $n\to\infty$, therefore $H_n(\xi)\to 0$ a.e. in $\mathbb{H}^N$ as $n\to\infty$. Consequently, the conclusion of Proposition \ref{prop4.3}-$(i)$ follows immediately by applying the Vitali's theorem. Furthermore, by using $(f1),(f3)$ and Corollary \ref{cor2.5}, one can easily see that for any $\vartheta_1,\vartheta_2\in\mathbb{R}$
\begin{align*}|f(\xi,\vartheta_1+\vartheta_2)(\vartheta_1+\vartheta_2)-f(\xi,\vartheta_1)\vartheta_1|&=\bigg|\int_{0}^{1}\frac{\mathrm{d}}{\mathrm{d}s}\big(f(\xi,\vartheta_1+s\vartheta_2)(\vartheta_1+s\vartheta_2)\big)~\mathrm{d}s\bigg|\\
&=\bigg|\int_{0}^{1}\big(\partial_u f(\xi,\vartheta_1+s\vartheta_2)\vartheta_2(\vartheta_1+s\vartheta_2)+f(\xi,\vartheta_1+s\vartheta_2)\vartheta_2\big)~\mathrm{d}s\bigg| \\
&\leq 2 \int_{0}^{1}\Big[\varepsilon|\vartheta_1+s\vartheta_2|^{Q-1}|\vartheta_2|+C_\varepsilon|\vartheta_2|\Phi(\alpha_0|\vartheta_1+s\vartheta_2|^{Q^\prime})\Big]~\mathrm{d}s\\
&\leq 2^{Q-1}\varepsilon\Big[|\vartheta_1|^{Q-1}|\vartheta_2|+|\vartheta_2|^{Q}\Big]+2C_\varepsilon|\vartheta_2|\Phi\big(\alpha_0(|\vartheta_1|+|\vartheta_2|)^{Q^\prime}\big).\end{align*}
Denote $\vartheta_1=u_n-u,\vartheta_2=u$, $v_n=u_n-u$ and $h_n(\cdot)=|f(\cdot,v_n+u)(v_n+u)-f(\cdot,v_n)v_n-f(\cdot,u)u|^{\boldsymbol{t}}$. Following a similar idea as in the proof of Proposition \ref{prop4.3}-$(i)$, we can easily verify the validity of Proposition \ref{prop4.3}-$(ii)$. Observe that
\begin{align}\label{eq4.10}
 \mathcal{G}(u_n)-\mathcal{G}(u_n-u)-\mathcal{G}(u)&=\displaystyle\int_{\mathbb{H}^N}\Bigg(\displaystyle\int_{\mathbb{H}^N}\frac{\big(F(\eta,u_n)-F(\eta,u_n-u)-F(\eta,u)\big)}{r(\eta)^\beta {d_K(\xi,\eta)}^\lambda}~\mathrm{d}\eta\Bigg)\frac{f(\xi,u_n)u_n}{r(\xi)^\beta} ~\mathrm{d}\xi \notag\\
 &\quad+\displaystyle\int_{\mathbb{H}^N}\Bigg(\displaystyle\int_{\mathbb{H}^N}\frac{F(\eta,u_n-u)}{r(\eta)^\beta {d_K(\xi,\eta)}^\lambda}~\mathrm{d}\eta\Bigg)\frac{\big(f(\xi,u_n)u_n-f(\xi,u_n-u)(u_n-u)-f(\xi,u)u\big)}{r(\xi)^\beta}~\mathrm{d}\xi \notag\\
 & \quad+\displaystyle\int_{\mathbb{H}^N}\Bigg(\displaystyle\int_{\mathbb{H}^N}\frac{F(\eta,u)}{r(\eta)^\beta {d_K(\xi,\eta)}^\lambda}~\mathrm{d}\eta\Bigg)\frac{\big(f(\xi,u_n)u_n-f(\xi,u_n-u)(u_n-u)-f(\xi,u)u\big)}{r(\xi)^\beta}~\mathrm{d}\xi \notag\\
 & \quad+\displaystyle\int_{\mathbb{H}^N}\Bigg(\displaystyle\int_{\mathbb{H}^N}\frac{F(\eta,u)}{r(\eta)^\beta {d_K(\xi,\eta)}^\lambda}~\mathrm{d}\eta\Bigg)\frac{f(\xi,u_n-u)(u_n-u)}{r(\xi)^\beta}~\mathrm{d}\xi \notag\\
 &\quad+\displaystyle\int_{\mathbb{H}^N}\Bigg(\displaystyle\int_{\mathbb{H}^N}\frac{F(\eta,u_n-u)}{r(\eta)^\beta {d_K(\xi,\eta)}^\lambda}~\mathrm{d}\eta\Bigg)\frac{f(\xi,u)u}{r(\xi)^\beta}~\mathrm{d}\xi.
\end{align}
Using the growth assumptions of $f$ and $F$, \eqref{eq4.4} and the boundedness of $\{u_n\}_n$ in $\mathbf{X}$, we can conclude that  $\{F(\cdot,u_n)\}_n,~\{F(\cdot,u_n-u)\}_n,~\{f(\cdot,u_n)u_n\}_n,$ and $\{f(\cdot,u_n-u)(u_n-u)\}_n$ are bounded in $L^{\boldsymbol{t}}(\mathbb{H}^N)$, where $\boldsymbol{t}=\frac{2Q}{2Q-2\beta-\lambda}$. In view of Theorem \ref{thm2.9} and Proposition \ref{prop4.3}-$(i)$, we have
\begin{align*}
\Bigg|\int_{\mathbb{H}^N}\Bigg(\displaystyle\int_{\mathbb{H}^N}\frac{\big(F(\eta,u_n)-F(\eta,u_n-u)-F(\eta,u)\big)}{r(\eta)^\beta {d_K(\xi,\eta)}^\lambda}\mathrm{d}\eta\Bigg)\frac{f(\xi,u_n)u_n}{r(\xi)^\beta} \mathrm{d}\xi\Bigg|&\leq \|F(\cdot,u_n)-F(\cdot,u_n-u)-F(\cdot,u)\|_{\boldsymbol{t}}\|f(\cdot,u_n)u_n\|_{\boldsymbol{t}} \\
&\to 0~\text{as}~n\to\infty.
\end{align*}
It follows that 
\begin{equation}\label{eq4.11}
\lim_{n\to\infty} \int_{\mathbb{H}^N}\Bigg(\displaystyle\int_{\mathbb{H}^N}\frac{\big(F(\eta,u_n)-F(\eta,u_n-u)-F(\eta,u)\big)}{r(\eta)^\beta {d_K(\xi,\eta)}^\lambda}~\mathrm{d}\eta\Bigg)\frac{f(\xi,u_n)u_n}{r(\xi)^\beta}~ \mathrm{d}\xi=0.    
\end{equation}
Similarly, by using Theorem \ref{thm2.9} and Proposition \ref{prop4.3}-$(ii)$, we obtain
\begin{equation}\label{eq4.12}
    \lim_{n\to\infty} \displaystyle\int_{\mathbb{H}^N}\Bigg(\displaystyle\int_{\mathbb{H}^N}\frac{F(\eta,u_n-u)}{r(\eta)^\beta {d_K(\xi,\eta)}^\lambda}~\mathrm{d}\eta\Bigg)\frac{\big(f(\xi,u_n)u_n-f(\xi,u_n-u)(u_n-u)-f(\xi,u)u\big)}{r(\xi)^\beta}~\mathrm{d}\xi =0
\end{equation}
and 
\begin{equation}\label{eq4.13}
     \lim_{n\to\infty} \displaystyle\int_{\mathbb{H}^N}\Bigg(\displaystyle\int_{\mathbb{H}^N}\frac{F(\eta,u)}{r(\eta)^\beta {d_K(\xi,\eta)}^\lambda}~\mathrm{d}\eta\Bigg)\frac{\big(f(\xi,u_n)u_n-f(\xi,u_n-u)(u_n-u)-f(\xi,u)u\big)}{r(\xi)^\beta}~\mathrm{d}\xi=0.
\end{equation}
By the hypothesis, we have that $f(\xi,u_n-u)(u_n-u)\to 0$  a.e. in $\mathbb{H}^N$ and $F(\xi,u_n-u)\to 0$  a.e. in $\mathbb{H}^N$ as $n\to\infty$. It follows that $f(\xi,u_n-u)(u_n-u)\rightharpoonup 0$  and $F(\xi,u_n-u)\rightharpoonup 0$ in $L^{\frac{2Q}{2Q-2\beta-\lambda}}(\mathbb{H}^N)$ as $n\to\infty$. Consequently, due to Theorem \ref{thm2.9}, we also have
$$ \displaystyle \int_{\mathbb{H}^N} \frac{F(\eta,u)}{r(\xi)^{\beta }~{d_K(\xi,\eta)}^\lambda r(\eta)^{\beta}}~\mathrm{d}\eta\in L^{\frac{2Q}{2\beta+\lambda}}(\mathbb{H}^N).$$
This shows that
\begin{equation}\label{eq4.14}
 \lim_{n\to\infty} \displaystyle\int_{\mathbb{H}^N}\Bigg(\displaystyle\int_{\mathbb{H}^N}\frac{F(\eta,u)}{r(\eta)^\beta {d_K(\xi,\eta)}^\lambda}~\mathrm{d}\eta\Bigg)\frac{f(\xi,u_n-u)(u_n-u)}{r(\xi)^\beta}~\mathrm{d}\xi=0.   
\end{equation}
Once more by Theorem \ref{thm2.9}, we can notice that
$$ \displaystyle \int_{\mathbb{H}^N} \frac{F(\eta,u_n-u)}{r(\xi)^{\beta }~{d_K(\xi,\eta)}^\lambda r(\eta)^{\beta}}~\mathrm{d}\eta\rightharpoonup 0~~\text{in}~~ L^{\frac{2Q}{2\beta+\lambda}}(\mathbb{H}^N)~\text{as}~n\to\infty. $$
This together with the fact that $f(\cdot,u)u\in L^{\frac{2Q}{2Q-2\beta-\lambda}}(\mathbb{H}^N)$, we obtain
\begin{equation}\label{eq4.15}
 \lim_{n\to\infty} \displaystyle\int_{\mathbb{H}^N}\Bigg(\displaystyle\int_{\mathbb{H}^N}\frac{F(\eta,u_n-u)}{r(\eta)^\beta {d_K(\xi,\eta)}^\lambda}~\mathrm{d}\eta\Bigg)\frac{f(\xi,u)u}{r(\xi)^\beta}~\mathrm{d}\xi=0.   
\end{equation}
Letting $n\to\infty$ in \eqref{eq4.10} and using the convergence results \eqref{eq4.11}--\eqref{eq4.15}, we obtain the validity of Proposition \ref{prop4.3}-$(a)$. In a similar way, we can prove Proposition \ref{prop4.3}-$(b)$. This completes the proof. 
\end{proof}
For $\upsilon$ defined as in $(f4)$ and $e\in\mathbf{X}$ as stated in Lemma \ref{lem3.3}, we define
 $$ M(e)= \displaystyle\int_{\mathbb{H}^N}\Bigg(\displaystyle\int_{\mathbb{H}^N}\frac{e(\eta)^\upsilon}{r(\eta)^\beta {d_K(\xi,\eta)}^\lambda}~\mathrm{d}\eta\Bigg)\frac{e(\xi)^\upsilon}{r(\xi)^\beta}~\mathrm{d}\xi. $$
 Due to Corollary \ref{cor2.3}, we can conclude that $e(\cdot)\in L^{\boldsymbol{\tau}\upsilon}(\mathbb{H}^N)$ for any $\boldsymbol{\tau}\in(Q,\infty)$. This fact together with Theorem \ref{thm2.9} implies at once that $M(e)$ is well-defined. Further, choose $\nu\geq 2Q^2 $ in \eqref{eq3.2} such that $\nu p> 2Q\upsilon$ holds. Let $\zeta,\zeta^\prime>1$ satisfying $\frac{1}{\zeta}+\frac{1}{\zeta^\prime}=1$, $\widehat{C}$ is defined as in $(f4)$ and $\mathcal{S}_{\frac{2Q\zeta\nu}{2Q-2\beta-\lambda},Q}$ is the best constant in the embedding $HW^{1,Q}(\mathbb{H}^N)\hookrightarrow L^{\frac{2Q\zeta\nu}{2Q-2\beta-\lambda}}(\mathbb{H}^N)$. Now, we denote $\gamma^\ast$ as follows:
 \begin{align}\label{eq4.999}
 \gamma^\ast=\max\bigg\{1, \Gamma_1^{\frac{\nu-Q}{Q}}, \big(\Gamma_2\Gamma_3\big)^{\frac{(2\upsilon-p)(\nu-Q)}{p\nu-2\upsilon Q}}\bigg\},    
 \end{align}
 where $$ \Gamma_1 \Gamma_2= \bigg(\frac{\alpha_Q}{2^{Q^\prime}}\bigg)^{1-Q},~\Gamma_2 \boldsymbol{\kappa}= \bigg(\boldsymbol{\wp}\widetilde{C}_2 \mathcal{S}_{\frac{2Q\zeta\nu}{2Q-2\beta-\lambda},Q}^{-\nu}\bigg)^{\frac{Q}{\nu-Q}},~~(\text{ the constants $\boldsymbol{\wp},\widetilde{C}_2>0$ are specified in Lemma \ref{lem4.5}}) $$
 and $$ \Gamma_3=\displaystyle\sum_{t\in\{p,Q\}} \bigg(\frac{1}{t}-\frac{1}{2\upsilon}\bigg)\bigg(\frac{\upsilon}{2 \widehat{C}^{2}}\bigg)^{\frac{t}{2\upsilon-t}}\frac{\|e\|^{\frac{2\upsilon t}{2\upsilon-t}}_{HW^{1,t}}}{\big(M(e)\big)^{\frac{t}{2\upsilon-t}}}. $$
Notice that whenever $\gamma>\gamma^\ast$, then we have
$$ \Gamma_2 \gamma ^{\frac{Q}{\nu-Q}}< \frac{\gamma ^{\frac{p}{2\upsilon-p}}}{\Gamma_3}.$$
Fix $c_0$ in such a way that there holds
\begin{align}\label{eq4.17}
  \Gamma_2 \gamma ^{\frac{Q}{\nu-Q}}<\frac{1}{c_0} <\frac{\gamma ^{\frac{p}{2\upsilon-p}}}{\Gamma_3}.  
\end{align}
This together with the definition of $\Gamma_2$ gives
\begin{align}\label{eq4.18}
    c_0< \boldsymbol{\kappa} \bigg(\boldsymbol{\wp}\widetilde{C}_2\gamma \mathcal{S}_{\frac{2Q\zeta\nu}{2Q-2\beta-\lambda},Q}^{-\nu}\bigg)^{\frac{Q}{Q-\nu}}.
\end{align}
Moreover, from the definition of $\Gamma_1$ and using the fact that $\gamma>\gamma^\ast$, one can easily obtain that 
\begin{align}\label{eq4.19}
  \bigg(\boldsymbol{\wp}\widetilde{C}_2\gamma \mathcal{S}_{\frac{2Q\zeta\nu}{2Q-2\beta-\lambda},Q}^{-\nu}\bigg)^{\frac{Q}{Q-\nu}} < \bigg(\frac{\alpha_Q}{2^{Q^\prime}}\bigg)^{Q-1}.
\end{align}
In view of \eqref{eq4.18} and \eqref{eq4.19}, we have
\begin{equation}\label{eq4.21}
    c_0<\boldsymbol{\kappa}\min\bigg\{\bigg(\frac{\alpha_Q}{2^{Q^\prime}}\bigg)^{Q-1},\bigg(\boldsymbol{\wp}\widetilde{C}_2\gamma \mathcal{S}_{\frac{2Q\zeta\nu}{2Q-2\beta-\lambda},Q}^{-\nu}\bigg)^{\frac{Q}{Q-\nu}}\bigg\}=\boldsymbol{\kappa} \bigg(\boldsymbol{\wp}\widetilde{C}_2\gamma \mathcal{S}_{\frac{2Q\zeta\nu}{2Q-2\beta-\lambda},Q}^{-\nu}\bigg)^{\frac{Q}{Q-\nu}}.
\end{equation}
The following lemma ensures the existence of a $(PS)_c$ condition for $J$ at a suitable level $c$ depending upon $c_0$.
\begin{lemma}\label{lem4.5}
 Suppose $\mu\in(0,\mu^\ast]$, where $\mu^\ast$ as in Lemma \ref{lem3.2} and $\gamma>\gamma^\ast$, with $\gamma^\ast$ is  defined as in \eqref{eq4.999}. Further, assume that  $c<c_0-\Theta\mu^{\vartheta}$, where $c_0>0$ satisfying \eqref{eq4.21} and $\Theta$ as in Lemma \ref{lem4.1}. Then, every $(PS)_c$ sequence for $J$ in $\mathbf{X}$ converges strongly to a function $w_{\mu,\gamma}\in\mathbf{X}$, which is a nontrivial solution of \eqref{main problem} at level $c$.
\end{lemma}
\begin{proof}
  Let $\{u_n\}_n$ be a $(PS)_c$ sequence for $J$ in $\mathbf{X}$, therefore we have  \begin{equation}\label{eq4.22}
  J(u_n)\to c~\text{and}~ J^\prime(u_n)\to 0~~\text{in}~\mathbf{X^\ast}~\text{as} ~n\to\infty. 
 \end{equation}
Now, by using Lemma \ref{lem4.1}, we obtain that $\{u_n\}_n$ is bounded in $\mathbf{X}$. Consequently, there exists $w_{\mu,\gamma}\in\mathbf{X}$ such that up to a subsequence, not relabelled, $u_n\rightharpoonup w_{\mu,\gamma}$ in $\mathbf{X}$ as $n\to\infty$. Following similar idea as in Lemma \ref{lem3.4}, we can prove $\{u_n\}_n$ and $w_{\mu,\gamma}$ are nonnegative in $\mathbf{X}$. In view of \eqref{eq4.21}, we can notice that \eqref{eq4.3} still holds. Using this fact together with $c<c_0-\Theta\mu^{\vartheta}$, we obtain that \eqref{eq3.17} holds. Therefore, by arguing similarly as in Proposition \ref{prop3.6}, we can infer that $D_H u_n\to D_H w_{\mu,\gamma}$ a.e. in $\mathbb{H}^N$ as $n\to\infty$. Consequently, by using Corollary \ref{cor2.3}, Lemma \ref{lem2.4} and the fact that the embedding 
     $\mathbf{X}\hookrightarrow HW^{1,t}(\mathbb{H}^N)$ is continuous for $t\in\{p,Q\}$, we obtain 
  \begin{equation}\label{eq4.23}
   \begin{cases}
      u_n\rightharpoonup w_{\mu,\gamma} ~\text{in}~HW^{1,t}(\mathbb{H}^N),u_n\to w_{\mu,\gamma} ~\text{in}~ L^s_g(\mathbb{H}^N),\\
       u_n\to w_{\mu,\gamma} ~\text{in}~L^\tau(B_R),~\text{for any}~R>0~\text{and}~\tau\in[1,\infty),\\
    
      u_n\to w_{\mu,\gamma} ~\text{a.e. in}~\mathbb{H}^N,~
      D_H u_n\to D_H w_{\mu,\gamma} ~\text{a.e. in}~\mathbb{H}^N,\\
       |D_H u_n|_H^{t-2}D_H u_n \rightharpoonup |D_H w_{\mu,\gamma}|_H^{t-2}D_H w_{\mu,\gamma}~\text{in}~L^{\frac{t}{t-1}}(\mathbb{H}^N,\mathbb{R}^{2N}),\\~  u_n^{t-1} \rightharpoonup w_{\mu,\gamma}^{t-1} ~\text{in}~L^{\frac{t}{t-1}}(\mathbb{H}^N)~\text{as}~n\to\infty.
      \end{cases}  
  \end{equation} 
From \eqref{eq4.23} and Br\'ezis-Lieb lemma, we deduce that 
\begin{equation}\label{eq4.24}   
        \|u_n-w_{\mu,\gamma}\|^t_{HW^{1,t}}=\|u_n\|^{t}_{HW^{1,t}}-\|w_{\mu,\gamma}\|^t_{HW^{1,t}}+o_n(1)~~\text{as}~n\to\infty,        
\end{equation}
for any $t\in\{p,Q\}$. Also, due to \eqref{eq4.23}, we can assume that up to a subsequence, still denoted by the same symbol such that $\|u_n\|_{HW^{1,p}}\to l_p$ and $\|u_n\|_{HW^{1,Q}}\to l_Q$ as $n\to\infty$, where $l_p,l_Q\geq 0$. In light of \eqref{eq4.24}, we have  
 $$\|u_n-w_{\mu,\gamma}\|^Q_{HW^{1,Q}}=\|u_n\|^{Q}_{HW^{1,Q}}-\|w_{\mu,\gamma}\|^Q_{HW^{1,Q}}+o_n(1)\leq \|u_n\|^{Q}_{HW^{1,Q}}+o_n(1)~\text{as}~n\to\infty. $$   
Letting $n\to\infty$ in the above inequality and using \eqref{eq4.444}, we get
    \begin{align}\label{eq4.25}
    \limsup_{n\to\infty}\|u_n-w_{\mu,\gamma}\|^{Q^\prime}_{HW^{1,Q}}&\leq \limsup_{n\to\infty}\|u_n\|^{Q^\prime}_{HW^{1,Q}}\leq \bigg(\frac{c+\Theta\mu^{\vartheta}}{\boldsymbol{\kappa}}\bigg)^{\frac{1}{Q-1}}<\frac{\alpha_Q}{2^{Q^\prime}}<\alpha_Q~~(\text{thanks to} ~\eqref{eq4.21}).\end{align}
Due to the Fatous lemma, we also have
\begin{equation}
    \liminf_{n\to\infty} \displaystyle\int_{\mathbb{H}^N}\Bigg(\displaystyle\int_{\mathbb{H}^N}\frac{F(\eta,u_n)}{r(\eta)^\beta {d_K(\xi,\eta)}^\lambda}~\mathrm{d}\eta\Bigg)\frac{f(\xi,u_n)}{r(\xi)^\beta} w_{\mu,\gamma}~\mathrm{d}\xi\geq \displaystyle\int_{\mathbb{H}^N}\Bigg(\displaystyle\int_{\mathbb{H}^N}\frac{F(\eta,w_{\mu,\gamma})}{r(\eta)^\beta {d_K(\xi,\eta)}^\lambda}~\mathrm{d}\eta\Bigg)\frac{f(\xi,w_{\mu,\gamma})}{r(\xi)^\beta} w_{\mu,\gamma}~\mathrm{d}\xi.
\end{equation}
Further, by the H\"older's inequality and \eqref{eq4.23}, we get
$$\bigg|\displaystyle\int_{\mathbb{H}^N} g(\xi) u_n^{s-1}(u_n-w_{\mu,\gamma})~\mathrm{d}\xi \bigg|\leq \|u_n\|^{s-1}_{s,g}\|u_n-w_{\mu,\gamma}\|_{s,g}\to 0~\text{as}~n\to\infty.$$
It follows that 
\begin{align}\label{eq4.27}
 \lim_{n\to\infty} \displaystyle\int_{\mathbb{H}^N} g(\xi) u_n^{s-1}(u_n-w_{\mu,\gamma})~\mathrm{d}\xi=0.   
\end{align}
By using \eqref{eq4.23} and \eqref{eq4.25} --\eqref{eq4.27}, we obtain
 \begin{align*}0&=\lim_{n\to\infty}\langle{J^\prime(u_n),u_n-w_{\mu,\gamma}}\rangle \\
 &=\sum_{t\in\{p,Q\}}\Bigg( \lim_{n\to\infty}\Bigg[\|u_n\|^t_{HW^{1,t}}-\int_{\mathbb{H}^N}\Big(|D_H u_n|_H^{t-2}(D_H u_n,D_H w_{\mu,\gamma})_H+u_n^{t-1}w_{\mu,\gamma}\Big)~\mathrm{d}\xi\Bigg]\Bigg)\\
 &\qquad-\mu \lim_{n\to\infty}\displaystyle\int_{\mathbb{H}^N} g(\xi) u_n^{s-1}(u_n-w_{\mu,\gamma})~\mathrm{d}\xi -\gamma \lim_{n\to\infty} \displaystyle\int_{\mathbb{H}^N}\Bigg(\displaystyle\int_{\mathbb{H}^N}\frac{F(\eta,u_n)}{r(\eta)^\beta {d_K(\xi,\eta)}^\lambda}~\mathrm{d}\eta\Bigg)\frac{f(\xi,u_n)}{r(\xi)^\beta}(u_n- w_{\mu,\gamma})~\mathrm{d}\xi\\
 &\geq \sum_{t\in\{p,Q\}} \big(l^t_t-\|w_{\mu,\gamma}\|^t_{HW^{1,t}}\big)-\gamma \lim_{n\to\infty} \big(\mathcal{G}(u_n)-\mathcal{G}(w_{\mu,\gamma})\big)~\mathrm{d}\xi.   \end{align*}
 Now by employing Proposition \ref{prop4.3} and \eqref{eq4.24}, we have
 \begin{align}\label{eq4.28} \|u_n-w_{\mu,\gamma}\|^p_{HW^{1,p}}+\|u_n-w_{\mu,\gamma}\|^Q_{HW^{1,Q}}\leq \gamma  \mathcal{G}(u_n-w_{\mu,\gamma})+o_n(1)~\text{as}~n\to\infty .\end{align}
Following similar arguments as in Proposition \ref{prop3.5}, we can deduce that there exists a constant $\widetilde{C}>0$ such that
\begin{equation}\label{eq4.29}
    \Bigg\| \displaystyle \int_{\mathbb{H}^N} \frac{F(\eta,u_n- w_{\mu,\gamma})}{r(\xi)^{\beta }~{d_K(\xi,\eta)}^\lambda r(\eta)^{\beta}}~\mathrm{d}\eta\Bigg\|_{\frac{2Q}{2\beta+\lambda}}\leq \widetilde{C}.
\end{equation}
It follows from \eqref{eq4.25} that there exist $m>0$ and $n_0\in\mathbb{N}$ such that $\|u_n-w_{\mu,\gamma}\|^{Q^\prime}_{HW^{1,Q}}<m<\alpha_Q$ for all $n\geq n_0$. Further, let $\nu\geq 2Q^2 $ in \eqref{eq3.2} and $\zeta,\zeta^\prime>1$ satisfying $\frac{1}{\zeta}+\frac{1}{\zeta^\prime}=1$. Choose $\alpha>\alpha_0$ close to $\alpha_0$, $\zeta^\prime$ close to $1$ in such a way that we still have $\displaystyle\frac{2\alpha Q \zeta^\prime }{2Q-2\beta-\lambda} \|u_n-w_{\mu,\gamma}\|^{Q^\prime}_{HW^{1, Q}}<m<\alpha_Q$ for all $n\geq n_0$. Therefore, from the H\"older's inequality, \eqref{eq3.2}, \eqref{eq4.29} , Corollary \ref{cor2.3} and Corollary \ref{cor2.5}, we obtain for all $n\geq n_0$ that
\begin{align}\label{eq4.30}
    \Big|\mathcal{G}(u_n-w_{\mu,\gamma})\Big|& \leq \widetilde{C} \|f(\cdot,u_n- w_{\mu,\gamma})(u_n- w_{\mu,\gamma})\|_{\frac{2Q}{2Q-2\beta-\lambda}} \notag\\
    & \leq \widetilde{C}\Bigg[\displaystyle\int_{\mathbb{H}^N}\Big[\varepsilon |u_n- w_{\mu,\gamma}|^{Q}+\widetilde{C}_\varepsilon |u_n- w_{\mu,\gamma}|^\nu\Phi(\alpha |u_n- w_{\mu,\gamma}|^{Q^\prime}) \Big]^{\frac{2Q}{2Q-2\beta-\lambda}}~\mathrm{d}\xi\Bigg]^{\frac{2Q-2\beta-\lambda}{2Q}} \notag\\
    &\leq C_{1,\varepsilon} \|u_n- w_{\mu,\gamma}\|^Q_{\frac{2Q^2}{2Q-2\beta-\lambda}}+ C_2 \Bigg[ \displaystyle\int_{\mathbb{H}^N} |u_n- w_{\mu,\gamma}|^{\frac{2Q\nu}{2Q-2\beta-\lambda}} \Phi\bigg(\frac{2\alpha Q}{2Q-2\beta-\lambda} |u_n- w_{\mu,\gamma}|^{Q^\prime}\bigg)~\mathrm{d}\xi\Bigg]^{\frac{2Q-2\beta-\lambda}{2Q}} \notag\\
    & \leq C_{1,\varepsilon} \mathcal{S}^{-Q}_{\frac{2Q^2}{2Q-2\beta-\lambda},Q} \|u_n- w_{\mu,\gamma}\|^Q_{HW^{1,Q}}+C_2 \|u_n- w_{\mu,\gamma}\|^\nu_{\frac{2Q\nu\zeta}{2Q-2\beta-\lambda}} \notag\\
    &\qquad\times \Bigg[ \displaystyle\int_{\mathbb{H}^N}  \Phi\bigg(\frac{2\alpha Q\zeta^\prime}{2Q-2\beta-\lambda} \|u_n-w_{\mu,\gamma}\|^{Q^\prime}_{HW^{1, Q}} \bigg|\frac{u_n- w_{\mu,\gamma}}{\|u_n-w_{\mu,\gamma}\|_{HW^{1, Q}}}\bigg|^{Q^\prime}\bigg)~\mathrm{d}\xi\Bigg]^{\frac{2Q-2\beta-\lambda}{2Q\zeta^\prime}} \notag\\
    & \leq C_{1,\varepsilon} \mathcal{S}^{-Q}_{\frac{2Q^2}{2Q-2\beta-\lambda},Q} \|u_n- w_{\mu,\gamma}\|^Q_{HW^{1,Q}}+\widetilde{C}_2 \|u_n- w_{\mu,\gamma}\|^\nu_{\frac{2Q\nu\zeta}{2Q-2\beta-\lambda}},
\end{align}
where $C_{1,\varepsilon}, C_2$ are positive constants, $\mathcal{S}_{\frac{2Q^2}{2Q-2\beta-\lambda},Q}$ is the best constant in the embedding  $HW^{1,Q}(\mathbb{H}^N)\hookrightarrow L^{\frac{2Q^2}{2Q-2\beta-\lambda}}(\mathbb{H}^N)$ and 
$$\widetilde{C}_2=C_2 \Bigg[ \sup_{n\geq n_0}\displaystyle\int_{\mathbb{H}^N}  \Phi\bigg(\frac{2\alpha Q\zeta^\prime}{2Q-2\beta-\lambda} \|u_n-w_{\mu,\gamma}\|^{Q^\prime}_{HW^{1, Q}} \bigg|\frac{u_n- w_{\mu,\gamma}}{\|u_n-w_{\mu,\gamma}\|_{HW^{1, Q}}}\bigg|^{Q^\prime}\bigg)~\mathrm{d}\xi\Bigg]^{\frac{2Q-2\beta-\lambda}{2Q\zeta^\prime}},$$
which is finite due to Theorem \ref{thm2.7}. Hence, from \eqref{eq4.28} and \eqref{eq4.30}, we get as $n\to\infty$ that
$$\widetilde{C}_2 \gamma \|u_n- w_{\mu,\gamma}\|^\nu_{\frac{2Q\nu\zeta}{2Q-2\beta-\lambda}}+o_n(1)\geq \|u_n-w_{\mu,\gamma}\|^p_{HW^{1,p}}+\bigg(1-C_1 \gamma \mathcal{S}^{-Q}_{\frac{2Q^2}{2Q-2\beta-\lambda},Q}\bigg)\|u_n-w_{\mu,\gamma}\|^Q_{HW^{1,Q}}. $$
Note that $\gamma>1$, $\varepsilon>0$ is arbitrary, and $C_{1,\varepsilon}$ is a positive constant having positive power of $\varepsilon$. Therefore, we can choose $\gamma\in \bigg(1,\frac{1}{C_{1,\varepsilon} \mathcal{S}^{-Q}_{\frac{2Q^2}{2Q-2\beta-\lambda},Q}\ }\bigg)$ and denote $\frac{1}{\boldsymbol{\wp}}=1-C_{1,\varepsilon} \gamma \mathcal{S}^{-Q}_{\frac{2Q^2}{2Q-2\beta-\lambda},Q}>0$, then for $n\to\infty$, we deduce from the above inequality that 
\begin{equation}\label{eq4.31}
  \widetilde{C}_2 \gamma \|u_n- w_{\mu,\gamma}\|^\nu_{\frac{2Q\nu\zeta}{2Q-2\beta-\lambda}}+o_n(1)\geq \|u_n-w_{\mu,\gamma}\|^p_{HW^{1,p}}+\frac{1}{\boldsymbol{\wp}}\|u_n-w_{\mu,\gamma}\|^Q_{HW^{1,Q}}.  
\end{equation}
Using the fact that the embedding $HW^{1,Q}(\mathbb{H}^N)\hookrightarrow L^{\frac{2Q\zeta\nu}{2Q-2\beta-\lambda}}(\mathbb{H}^N)$ is continuous, we get
\begin{equation}\label{eq4.32}
  \|u_n- w_{\mu,\gamma}\|^Q_{\frac{2Q\nu\zeta}{2Q-2\beta-\lambda}}\leq  \mathcal{S}_{\frac{2Q\zeta\nu}{2Q-2\beta-\lambda},Q}^{-Q} \|u_n-w_{\mu,\gamma}\|^Q_{HW^{1,Q}}.
\end{equation}
Combining \eqref{eq4.31} and \eqref{eq4.32}, we get as $n\to\infty$ that
$$ \widetilde{C}_2 \gamma \|u_n- w_{\mu,\gamma}\|^\nu_{\frac{2Q\nu\zeta}{2Q-2\beta-\lambda}}+o_n(1)\geq \frac{1}{\boldsymbol{\wp}\mathcal{S}_{\frac{2Q\zeta\nu}{2Q-2\beta-\lambda},Q}^{-Q} } \|u_n- w_{\mu,\gamma}\|^Q_{\frac{2Q\nu\zeta}{2Q-2\beta-\lambda}}. $$
Denote $\displaystyle\lim_{n\to\infty} \|u_n- w_{\mu,\gamma}\|_{\frac{2Q\nu\zeta}{2Q-2\beta-\lambda}}=L\geq 0$. In order to complete the proof of Lemma \ref{lem4.5}, it is sufficient to show that $L=0$. Indeed, if not, let $L>0$, then sending $n\to\infty$ in the above inequality, we conclude that
\begin{equation}\label{eq4.33}
    L\geq \Bigg(\frac{1}{\boldsymbol{\wp} \widetilde{C}_2 \gamma\mathcal{S}_{\frac{2Q\zeta\nu}{2Q-2\beta-\lambda},Q}^{-Q} }\Bigg)^{\frac{1}{\nu-Q}}.
\end{equation}
In view of \eqref{eq4.444}, \eqref{eq4.24} and \eqref{eq4.32}, we get as $n\to\infty$  
$$ c+\Theta\mu^{\vartheta}+o_n(1)\geq \boldsymbol{\kappa} \|u_n-w_{\mu,\gamma}\|^{Q}_{HW^{1,Q}}\geq \frac{\boldsymbol{\kappa}}{\mathcal{S}_{\frac{2Q\zeta\nu}{2Q-2\beta-\lambda},Q}^{-Q} } \|u_n- w_{\mu,\gamma}\|^Q_{\frac{2Q\nu\zeta}{2Q-2\beta-\lambda}}.  $$
Letting $n\to\infty$ in the above inequality and using the hypothesis along with \eqref{eq4.21} and \eqref{eq4.33}, we have
$$c_0> \frac{\boldsymbol{\kappa}L^Q}{\mathcal{S}_{\frac{2Q\zeta\nu}{2Q-2\beta-\lambda},Q}^{-Q} } \geq \boldsymbol{\kappa} \bigg(\boldsymbol{\wp}\widetilde{C}_2\gamma \mathcal{S}_{\frac{2Q\zeta\nu}{2Q-2\beta-\lambda},Q}^{-\nu}\bigg)^{\frac{Q}{Q-\nu}}>c_0 ,$$
which is a contradiction. It follows that $L=0$. Consequently, sending $n\to\infty$ in \eqref{eq4.31}, we infer that $ u_n\to w_{\mu,\gamma} ~\text{in}~HW^{1,t}(\mathbb{H}^N)$ as $n\to\infty$ for $t\in\{p,Q\}$. Hence, we deduce that $u_n\to w_{\mu,\gamma}$ in $\mathbf{X}$ as $n\to\infty$. Moreover, we get $J(w_{\mu,\gamma})=c>0$ and $J^\prime(w_{\mu,\gamma})=0$. This shows that $w_{\mu,\gamma}\neq 0$ and $w_{\mu,\gamma}$ is a nontrivial nonnegative solution of \eqref{main problem}. This completes the proof.
\end{proof}

Suppose that $\mu\in(0,\mu^\ast]$, where $\mu^\ast$ as in Lemma \ref{lem3.2} and $\gamma>0$ be fixed. Observe that the mountain pass geometrical structures of the energy functional $J$ are satisfied, thanks to Lemma \ref{lem3.2} and Lemma \ref{lem3.3}.  Define a mountain pass level $c_{\mu,\gamma}$ for $J$ by
\begin{equation}\label{eq4.34}   c_{\mu,\gamma}=\inf_{g\in\Lambda}\max_{s\in[0,1] } J(g(s)),
\end{equation}
where $\Lambda=\{g\in\mathcal{C}([0,1],\mathbf{X}):~g(0)=0,~J(g(1))<0\}$. Note that $c_{\mu,\gamma}>0$. To apply the mountain pass theorem, which provides us the second independent solution for \eqref{main problem}, we shall show that $c_{\mu,\gamma}$ falls into the range of validity of the $(PS)_c$ condition given in Lemma \ref{lem4.5}.
\begin{lemma}\label{lem4.6}
Assume that $\gamma>\gamma^\ast$, where $\gamma^\ast$ is defined as in \eqref{eq4.999}. Then, there exists $\tilde{\mu}\in (0,\mu^\ast]$, depending on $\gamma$, such that $c_{\mu,\gamma}<c_0-\Theta\mu^{\vartheta}$ for any $\mu\in(0,\tilde{\mu}]$, with $c_0$ satisfying \eqref{eq4.17} and $\Theta$ as stated in Lemma \ref{lem4.1}.
\end{lemma}
\begin{proof}
Let $\gamma>\gamma^\ast$ be fixed and $e\in\mathbf{X}$ as in Lemma \ref{lem3.3}. Choose $\Bar{\mu}\in(0,\mu^\ast]$ sufficiently small such that $c_0\geq 2\Theta\mu^{\vartheta}$ holds for all $\mu\in(0,\Bar{\mu}]$. One can notice that $J(\ell e)\to 0$ as $\ell\to 0^+$. Therefore, we can choose $\ell_0$ very small such that 
\begin{equation}\label{eq4.36}
    \sup_{\ell\in[0,\ell_0]} J(\ell e)<\frac{c_0}{2}\leq c_0-\Theta\mu^{\vartheta}~\text{for all}~\mu\in(0,\Bar{\mu}].
\end{equation}
Observe that $\gamma>1$, $Q>p$ and hence by using $(f4)$, we have
\begin{align*} \sup_{\ell\geq \ell_0} J(\ell e)&\leq \sup_{\ell\geq \ell_0}\bigg(\sum_{t\in\{p,Q\}}\Bigg[\frac{\ell^t}{t}\|e\|^{t}_{HW^{1,t}}\Bigg]-\frac{\mu \ell^s}{s}\|e\|^s_{s,g}-\frac{2\gamma \widehat{C}^2 \ell^{2\upsilon}}{\upsilon^2}M(e)\bigg)\\
&\leq \sum_{t\in\{p,Q\}}\Bigg[\max_{\ell\geq 0}\bigg(\frac{\ell^t}{t}\|e\|^{t}_{HW^{1,t}}-\frac{\gamma \widehat{C}^2 \ell^{2\upsilon}}{\upsilon^2}M(e)\Bigg] -\frac{\mu \ell_0^s}{s}\|e\|^s_{s,g}\\
&= \displaystyle\sum_{t\in\{p,Q\}} \bigg(\frac{1}{t}-\frac{1}{2\upsilon}\bigg)\bigg(\frac{\upsilon}{2\gamma \widehat{C}^{2}}\bigg)^{\frac{t}{2\upsilon-t}}\frac{\|e\|^{\frac{2\upsilon t}{2\upsilon-t}}_{HW^{1,t}}}{\big(M(e)\big)^{\frac{t}{2\upsilon-t}}} -\frac{\mu \ell_0^s}{s}\|e\|^s_{s,g} \\
&< {\gamma}^{-\frac{p}{2\upsilon-p}}\Bigg(\displaystyle\sum_{t\in\{p,Q\}} \bigg(\frac{1}{t}-\frac{1}{2\upsilon}\bigg)\bigg(\frac{\upsilon}{2 \widehat{C}^{2}}\bigg)^{\frac{t}{2\upsilon-t}}\frac{\|e\|^{\frac{2\upsilon t}{2\upsilon-t}}_{HW^{1,t}}}{\big(M(e)\big)^{\frac{t}{2\upsilon-t}}}\Bigg)-\frac{\mu \ell_0^s}{s}\|e\|^s_{s,g}\\
&<c_0-\frac{\mu \ell_0^s}{s}\|e\|^s_{s,g},
\end{align*}
we thank to \eqref{eq4.17}. Next, we choose $\Bar{\Bar{\mu}}\in(0,\mu^\ast]$ very small such that
  $$-\frac{\mu \ell_0^s}{s}\|e\|^s_{s,g}<-\Theta\mu^{\vartheta}~\text{for all}~\mu\in(0,\Bar{\Bar{\mu}}].$$  
It follows that
\begin{equation}\label{eq4.38}
 \sup_{\ell\geq \ell_0} J(\ell e)< c_0-\Theta\mu^{\vartheta}~\text{for all}~\mu\in(0,\Bar{\Bar{\mu}}].  
\end{equation}
Take $\tilde{\mu}=\min\{\Bar{\mu},\Bar{\Bar{\mu}}\}$ and the path $g(\ell)=\ell e$ belongs to $\Lambda$, where $\ell\in[0,1]$. Therefore, from \eqref{eq4.36} and \eqref{eq4.38}, we get
  $$ c_{\mu,\gamma}\leq\max_{\ell\in[0,1] } J(g(\ell))\leq \max_{\ell\geq 0 } J(g(\ell))< c_0-\Theta\mu^{\vartheta}~\text{for all}~\mu\in(0,\tilde{\mu}],$$    
and thus we conclude the proof.
\end{proof}
\begin{proof}[\textbf{Proof of Theorem \ref{thm1.2}}.]
The energy functional $J$ satisfies the mountain pass geometric structures for all $\mu\in(0,\mu^\ast]$ and for all $\gamma>0$, thanks to Lemma \ref{lem3.2} and Lemma \ref{lem3.3}. Moreover, by Lemma \ref{lem4.5} and Lemma \ref{lem4.6}, we obtain that for $\gamma>\gamma^\ast$, there exists $\tilde{\mu}\leq {\mu}^\ast$ such that $J$ admits a critical point $w_{\mu,\gamma}\in\mathbf{X}$ at the mountain pass level $c_{\mu,\gamma}>0$, estimated as in Lemma \ref{lem4.6} for all $\mu\in(0,\tilde{\mu}]$. Consequently, by Lemma \ref{lem4.5}, we have $w_{\mu,\gamma}\in\mathbf{X}$ is a nontrivial nonnegative solution of \eqref{main problem}, which is independent of the solution $u_{\mu,\gamma}$ of \eqref{main problem} obtained in Theorem \ref{thm1.1}, because $J(u_{\mu,\gamma})=m_{\mu,\gamma}<0<c_{\mu,\gamma}=J(w_{\mu,\gamma})$. This completes the proof of Theorem \ref{thm1.2}.
\end{proof}
\section*{Acknowledgements}
The first author wishes to convey his sincere appreciation for the DST INSPIRE Fellowship with reference number DST/INSPIRE/03/2019/000265 funded by the Government of India. The second author acknowledges the support provided by the Start-up Research Grant from DST-SERB with sanction no. SRG/2022/000524. The third author was supported by the DST-INSPIRE Grant DST/INSPIRE/04/2018/002
208 sponsored by the Government of India. The fourth author received assistance from the UGC Grant with reference no. 191620169606 funded by the Government of India.

\bibliography{ref}
\bibliographystyle{abbrv}
\Addresses
\end{document}